\documentclass[11pt, a4paper, leqno]{amsart}
\usepackage[utf8]{inputenc}
\usepackage{amsmath}
\usepackage{amsfonts}
\usepackage{amssymb}
\usepackage{times}
\usepackage[T1]{fontenc} 
\usepackage{url} 
\usepackage{esint}
\usepackage[hidelinks]{hyperref} 
\usepackage{graphicx} 
\usepackage{enumitem}
\IfFileExists{ulem.sty}{\usepackage[normalem]{ulem}}{} 
\usepackage{cancel}
\usepackage{tabularx,longtable,booktabs,microtype}
    \usepackage{mathtools}
\usepackage{comment}
\usepackage{subcaption}
\IfFileExists{ulem.sty}{\usepackage[normalem]{ulem}}{}

\let\oldmarginpar\marginpar
\renewcommand\marginpar[1]{\-\oldmarginpar[\raggedleft\footnotesize #1]%
{\raggedright\footnotesize #1}}

\usepackage{amsthm}

\numberwithin{equation}{section}

\theoremstyle{plain}
\newtheorem{theorem}{Theorem}[section]
\newtheorem{lemma}[theorem]{Lemma}

\newtheorem{corollary}[theorem]{Corollary}

\newtheorem*{bhfirst}{Theorem A}
\newtheorem*{bhhigher}{Theorem B}

\theoremstyle{definition}

\theoremstyle{remark}
\newtheorem{remark}[theorem]{Remark}

\def\le{\leqslant}
\def\leq{\leqslant}
\def\ge{\geqslant}
\def\geq{\geqslant}
\def\phi{\varphi}
\def\rho{\varrho}
\def\vartheta{\theta}

\date{\today}

\hypersetup{pdftitle={The Obstacle Problem Arising from the American Chooser Option},pdfauthor={Gugyum Ha, Junkee Jeon, and Jihoon Ok}}

\begin{document}

\title[The Obstacle Problem Arising from the American Chooser Option]{The Obstacle Problem Arising from the American Chooser Option}

\author{Gugyum Ha}
\address{Gugyum Ha, Department of Mathematics, Sogang University, Seoul 04107, Republic of Korea}
\email{\texttt{ggha@sogang.ac.kr}}

\author{Junkee Jeon}
\address{Junkee Jeon, Department of  Applied Mathematics, Kyung Hee University, Yongin-Si 17104, Republic of Korea}
\email{\texttt{junkeejeon@khu.ac.kr}}

\author{Jihoon Ok}
\address{Jihoon Ok, Department of Mathematics, Sogang University, Seoul 04107, Republic of Korea}
\email{\texttt{jihoonok@sogang.ac.kr}}


\begin{abstract}
{
We study the finite-horizon obstacle problem associated with the American
chooser option. Its obstacle is the maximum of two independently determined
American call and put values and has a spatial kink where these values
cross. Using truncation, penalization, and comparison, we construct a local
strong Sobolev solution and prove uniqueness within an exponential-growth
class. We show that contact occurs only on the smooth exercise-payoff
branches and that both contact regions have nonempty sections at every
time. We then establish strict monotonicity, terminal limits, and local
Lipschitz continuity of the two free boundaries. Localized one-phase
equations, quantitative nondegeneracy, and parabolic boundary Harnack
estimates yield $C^\infty$ regularity on the open time interval.
}
\end{abstract}

\maketitle

\section{Introduction}\label{sec:introduction}

{
An American chooser permits its holder to select an American call or an
American put at a stopping time no later than the choosing date $T$.
We study the associated lower obstacle problem in the Black--Scholes model:
}
\begin{equation}\label{eqV*}
    \begin{cases}
    {V^{\rm ch}(t,s)\geq\max\{C^A(t,s),P^A(t,s)\}\quad\text{in }D_T,}\\
        \partial_t V^{\rm ch}(t,s) + \mathfrak{L} V^{\rm ch}(t,s) \le 0  \quad  \text{for }\ (t,s)\in D_T\  \text{ with } \  V^{\rm ch}(t,s) = \max\{C^A(t,s),P^A(t,s)\},\\
        \partial_t V^{\rm ch}(t,s) + \mathfrak{L} V^{\rm ch}(t,s) = 0 \quad \text{for }\ (t,s)\in D_T\ \text{ with } \   V^{\rm ch}(t,s) > \max\{C^A(t,s),P^A(t,s)\},\\
        V^{\rm ch}(T,s) = \max\{C^A(T,s),P^A(T,s)\} \quad  \text{for } \ 0<s<\infty.
    \end{cases}
\end{equation}
where the differential operator $\mathfrak{L}$ is defined as
\[ \mathfrak{L}:=\frac{\sigma^2}{2}s^2\partial_{ss}+(r-q)s\partial_s-r, \quad \text{with} \quad r,\;\sigma>0,\;q\geq 0, \]
$D_\eta:=\{(t,s)\;:\; 0<t< \eta,\  0<s<\infty \}$ for $\eta>0$, and
the functions $C^A$ and $P^A$ satisfy the following obstacle problems,
respectively:
\begin{equation}\label{eqC*}
    \begin{cases}
    {C^A(t,s)\geq(s-K_c)^+\quad\text{in }D_{T_c},}\\
        \partial_t C^A (t,s) + \mathfrak{L} C^A (t,s) \le 0 \quad
         \text{for }\ (t,s)\in D_{T_c} \ \text{ with }\  C^A(t,s)=(s-K_c)^+,\\
        \partial_t C^A (t,s) + \mathfrak{L} C^A (t,s) = 0 \quad \text{for }\ (t,s)\in D_{T_c} \ \text{ with }\  C^A(t,s)>(s-K_c)^+,\\
        {C^A(T_c,s)} = (s-K_c)^+  \quad \text{for }\ 0<s<\infty.
    \end{cases}
\end{equation}
\begin{equation}\label{eqP*}
 \begin{cases}
    {P^A(t,s)\geq(K_p-s)^+\quad\text{in }D_{T_p},}\\
        \partial_t P^A (t,s) + \mathfrak{L} P^A (t,s) \le 0  \quad \text{for }\ (t,s)\in D_{T_p} \ \text{ with }\  P^A(t,s)=(K_p-s)^+,\\
        \partial_t P^A (t,s) + \mathfrak{L} P^A (t,s) = 0    \quad \text{for }\ (t,s)\in D_{T_p} \ \text{ with }\   P^A(t,s)>(K_p-s)^+,\\
        {P^A(T_p,s)} = (K_p-s)^+  \quad \text{for }\ 0<s<\infty.
    \end{cases}
\end{equation}
{We assume that}
\begin{equation*}
    0<T < \min\{T_c,T_p\}, \quad 0<K_p<K_c,\quad {q>0}.
\end{equation*}
{The standard American call and put values provide strong
solutions of \eqref{eqC*} and \eqref{eqP*}; see \cite{ZF09,ZF09-1}.}

{
The chooser contract and its two-boundary exercise structure are already
established in the literature. Detemple and Emmerling~\cite{DETEMPLE09}
analyze the exercise region, derive an early-exercise-premium representation
and coupled boundary integral equations, and also consider complex choosers
with unequal strikes and maturities. Qiu and Mitra~\cite{Qiu2018} study
the finite-horizon problem considered here. Their Theorem~2 proves the
stopping-region inclusions, Theorem~4 gives boundary monotonicity and
terminal limits, and Theorem~5 establishes continuity of the boundaries
and smooth fit. They also formulate the associated free-boundary problem.
Thus, the contract, its two-boundary geometry, these inclusions, and the
qualitative conclusions just listed are not new results of this paper.

Our aim is to construct the chooser value as a local strong Sobolev solution
and to develop a PDE proof of its boundary properties, including higher
regularity. The analytical distinction is the obstacle: the American call
and put are first determined independently, and their maximum then becomes
the obstacle for the chooser. There is no feedback from the chooser equation
into the two underlying option problems. Although each component has the
usual American-option regularity, their maximum has a spatial kink where
they cross. The construction must control this kink and justify why contact
with the chooser value occurs only on the smooth exercise-payoff branches.

The methodological starting point is Jeon and Oh~\cite{JJ19}, who construct
a strong solution and analyze two free boundaries for the American strangle
with an explicit payoff. We adapt their truncation, penalization, and
comparison strategy to the obstacle generated by two previously solved
American-option problems. The comparison principle and the regularity
estimates used below are standard PDE techniques. Our contribution lies
in verifying their applicability to this obstacle and in carrying the
argument through to strict boundary monotonicity and higher regularity.
Related PDE analyses of other American-option problems include
\cite{Jiang2002,Rapuch2005,KarlsenWallin2009,QinChenLaiYu2015}; these works
provide context rather than a claim that PDE formulations of chooser
options were previously unavailable.

We now relate the principal results to this literature.
All boundary statements in this paper use the forward variable
$\tau=T-t$ and the log-price $x=\log s$, so monotonicity directions are
reversed relative to calendar time.

Theorem~\ref{thm:strong-solution} constructs a strong solution for the
implicit maximum obstacle, and its uniqueness is explicitly restricted
to an exponential growth class. Jeon and Oh~\cite{JJ19} carried out such
a construction for the strangle, whose payoff is explicit, and valuation
and free boundary formulations for the chooser were given in
\cite{DETEMPLE09,Qiu2018}.

Several of our results give new PDE proofs of known statements.
Lemmas~\ref{FB1} and~\ref{FB: contact} prove the geometry of the
exercise region obtained in \cite{DETEMPLE09} and the stopping region
inclusion of Qiu and Mitra~\cite[Theorems~1 and~2]{Qiu2018}. The
explicit barriers in Lemma~\ref{FB: contact} also show that both
boundaries are finite on every time slice.
Theorems~\ref{thm:free-p}(ii) and~\ref{thm:free-c}(ii) recover, by new
PDE arguments, the terminal limit formulas of
\cite[Theorem~4]{Qiu2018} and the corresponding boundary results of
Detemple and Emmerling~\cite{DETEMPLE09}.

The remaining boundary results strengthen earlier conclusions.
Qiu and Mitra~\cite[Theorem~4]{Qiu2018} state weak monotonicity of the
boundaries, and Theorems~\ref{thm:free-p}(i) and~\ref{thm:free-c}(i)
establish strict monotonicity by comparison and the strong maximum
principle. Qiu and Mitra~\cite[Theorem~5]{Qiu2018} prove continuity of
the boundaries and smooth fit, and Jeon and Oh~\cite{JJ19} analyze the
boundaries of the strangle. Lemma~\ref{DV} and
Theorem~\ref{Lipschitz free boundary} give spatial smooth fit within the
Sobolev framework and improve continuity to local Lipschitz regularity
in time for the chooser boundaries. Finally,
Theorem~\ref{C1alpha free boundary} and
Corollary~\ref{smooth free boundary} prove $C^{1,\alpha}_{\rm loc}$
regularity for $0<\alpha<1/2$ and $C^\infty$ regularity on $(0,T)$.
Their counterparts are the continuity and smooth fit in \cite{Qiu2018}
and the smoothness of the strangle boundaries in \cite{JJ19}, where the
payoff is explicit. The proofs use localized one-phase equations and
boundary Harnack estimates.

In particular, the stopping-region inclusion is a known result whose
PDE proof is essential to our regularity argument: it identifies the local
obstacle with a smooth payoff branch and gives a strictly signed source
after subtraction. The boundary regularity theory itself is drawn from
the existing PDE literature; the issue addressed here is its application
to the chooser problem with its implicit obstacle. No regularity at the
terminal kink is included in the $C^\infty((0,T))$ conclusion.

Section~\ref{sec:2} gives the financial formulation and the properties of
the underlying American options. Section~\ref{sec:3} constructs the strong
solution, specifies its uniqueness class, and proves derivative estimates.
Section~\ref{sec:4} establishes the geometry, strict monotonicity, endpoint
limits, and regularity of the free boundaries. The appendices identify the
constructed solution with the stopping value and supply auxiliary proofs.\label{rev:intro-end}
}

\section{Preliminaries}\label{sec:2}

\subsection{Notations}

{
Throughout the paper, for each $\varepsilon>0$, let
$\varphi_\varepsilon\in C^\infty(\mathbb R)$ satisfy
\begin{equation}\label{phiepsilon}
\begin{aligned}
&\varphi_\varepsilon\geq0,\qquad
0\leq\varphi_\varepsilon'\leq1,\qquad
\varphi_\varepsilon''\geq0,\\
&\varphi_\varepsilon(\xi)=0
\quad\text{for }\xi\leq-\varepsilon,
\qquad
\varphi_\varepsilon(\xi)=\xi
\quad\text{for }\xi\geq\varepsilon,\\
&\varphi_\varepsilon(\xi)\leq(\xi+\varepsilon)^+
\quad\text{for all }\xi\in\mathbb R,\\
&\sup_{\xi\in\mathbb R}
\big|\varphi_\varepsilon(\xi)-\xi^+\big|
\longrightarrow0
\quad\text{as }\varepsilon\to0.
\end{aligned}
\end{equation}
}

{
Let $\mathcal{D}\subset\mathbb{R}^2$ be a bounded parabolic domain,
and denote by $\partial_p\mathcal D$ its parabolic boundary.
In particular, if
\[
    \mathcal D=(t_1,t_2)\times(a,b),
\]
then
\[
    \partial_p\mathcal D
    =
    \big(\{t_1\}\times[a,b]\big)
    \cup
    \big([t_1,t_2)\times\{a,b\}\big).
\]
}

\begin{itemize}

    \item {
    For $\alpha\in(0,1)$,
    $C^{\frac{\alpha}{2},\alpha}(\mathcal{D})$ denotes the
    parabolic H\"older space equipped with the norm
    \[
    \begin{aligned}
    \Vert V\Vert_{C^{\frac{\alpha}{2},\alpha}(\mathcal{D})}
    &:=
    \sup_{(t,x)\in\mathcal{D}}|V(t,x)|
    \\
    &\quad+
    \sup_{\substack{(t_1,x_1),(t_2,x_2)\in\mathcal{D}\\
    (t_1,x_1)\neq(t_2,x_2)}}
    \frac{|V(t_1,x_1)-V(t_2,x_2)|}
    {|t_1-t_2|^{\frac{\alpha}{2}}+|x_1-x_2|^\alpha}.
    \end{aligned}
    \]
    }

    \item
    {
    For the boundary Harnack estimates used later, we also adopt the
    parabolic H\"older notation of \cite{Kukuljan2022}. For
    $\alpha\in(0,1)$, $C_p^\alpha$ agrees with
    $C^{\frac{\alpha}{2},\alpha}$ defined above. More generally, spatial
    and temporal derivatives are assigned parabolic orders one and two,
    respectively; namely,
    \[
        D_p^k V
        :=
        \left\{
        \partial_t^\ell\partial_x^m V:
        2\ell+m=k
        \right\}.
    \]
    The spaces $C_p^\beta$, $\beta\geq1$, and parabolic domains of class
    $C_p^\beta$ are understood in the sense of
    \cite[Definitions~1.1 and~2.2]{Kukuljan2022}.
    }

    \item 
    $L^p(\mathcal{D}),p\geq1$ is the completion of $C^\infty(\mathcal{D})$
    under the following norm for $V$:
    \[
    \Vert V \Vert_{L^p(\mathcal{D})}:=
    \left(\iint_\mathcal{D} {\vert V(t,x)\vert}^p \,dx\, dt\right)^{\frac{1}{p}}.
    \]

    \item 
    $W^{1,2}_p(\mathcal{D}),$ $p\geq1$ is the completion of $C^{\infty}(\mathcal{D})$
    under the following norm for $V$:
    \[
    \Vert V \Vert_{W^{1,2}_p(\mathcal{D})}
    :=\left(
    \iint_\mathcal{D}
    \big\{
    \vert V \vert^p
    +\vert \partial_t V \vert^p
    +\vert \partial_x V \vert^p
    +\vert \partial^2_x V \vert^p
    \big\}
    \,dx\,dt
    \right)^{\frac{1}{p}}.
    \]

\end{itemize}

\subsection{Financial Background: American Chooser Option}


Let $(\Omega, \mathcal{G}, \mathbb{P})$ be a filtered probability space with a filtration $\{\mathcal{G}_t\}_{t \geq 0}$ satisfying the usual conditions. Under the risk-neutral measure $\mathbb{Q}$, the stock price $S_t$ follows a geometric Brownian motion (GBM) given by
\begin{equation*}
    dS_t = (r - q) S_t dt + \sigma S_t dW_t^{\mathbb{Q}}, \quad S_0 > 0,
\end{equation*}
where $r>0$ is the constant risk-free interest rate, $q \ge 0$ is the continuous dividend rate, $\sigma > 0$ is the volatility of the stock, and $W_t^{\mathbb{Q}}$ is a standard Brownian motion under the risk-neutral measure $\mathbb{Q}$.
 Throughout this paper, we assume $q > 0$ for simplicity. The case $q = 0$ leads to a degenerate situation with only one free boundary, but similar results can be obtained (see \cite{JJ19}).

To define the American chooser option, we first introduce the American call and put options. These two options serve as the fundamental building blocks for the American chooser option.
An \textit{American call option} with strike price $K_c$ and expiration time $T_c$ gives the holder the right to buy the underlying asset at price $K_c$ at any time $\theta \in \mathcal{U}_{t,T_c}$, where $\mathcal{U}_{t,T}$ is the set of all $\mathcal{G}$-stopping times taking values in $[t,T]$. The value $C^A(t,S_t)$ of the American call option at time $t$ is given by
\begin{equation}\label{def:CA}
    C^A(t, S_t) = \sup_{\theta \in \mathcal{U}_{t,T_c}} \mathbb{E}^{\mathbb{Q}} \left[ e^{-r(\theta - t)} (S_\theta - K_c)^+ \mid \mathcal{G}_t \right].
\end{equation}
Similarly, an \textit{American put option} with the strike price $K_p<K_c$ and expiration time $T_p$ grants the holder the right to sell the underlying asset at price $K_p$ at any time $\tau \in [0,T_p]$. The value $P^A(t,S_t)$ of the American put option at time $t$ is given by
\begin{equation}\label{def:PA}
    P^A(t, S_t) = \sup_{\theta \in \mathcal{U}_{t,T_p}} \mathbb{E}^{\mathbb{Q}} \left[ e^{-r(\theta - t)} (K_p - S_\theta)^+ \mid \mathcal{G}_t \right].
\end{equation}
If \( C^A \) and \( P^A \) are strong solutions to the obstacle problems \eqref{eqC*} and \eqref{eqP*}, respectively, then one can show, by applying It\^o's lemma for Sobolev spaces (see \cite{KRY80}), that they solve the corresponding optimal stopping problems \eqref{def:CA} and \eqref{def:PA}; see, for instance, \cite[Appendix B]{CYW11}.


Under these circumstances, the value of the American chooser option, denoted by $V^{\rm ch}$, is defined as the solution to the following optimal stopping problem:
\begin{equation}\label{V_ch}
    V^{\rm ch}(t,S_t) = \sup_{\theta \in \mathcal{U}_{t,T}}\mathbb{E}\left[e^{-r(\theta-t)} \max\left\{C^A(\theta,S_{\theta}),P^A(\theta,S_{\theta})\right\}\mid \mathcal{G}_t\right].
\end{equation}
{The stopping problem motivates \eqref{eqV*}. We first
construct its strong solution in Section~\ref{sec:3} and then verify its
stochastic representation in Appendix~\ref{Appendix.B}; see also
Appendix~\ref{Appendix.A} for the order of the argument.} 

\subsection{Change of variables}
We transform them into forward non-degenerate problems and denote
the solutions by $\widehat{C}^A$ and $\widehat{P}^A$, respectively. More precisely, we introduce the following transformation:
\begin{equation*}
    \widehat{C}^A(\zeta, x) = C^A(T_c-\zeta, e^x) \quad \text{and} \quad 
    \widehat{P}^A(\zeta, x) = P^A(T_p-\zeta, e^x). 
\end{equation*}
Moreover, we define the domain $\Omega_{\eta}$
for a given constant $\eta>0$ as
\begin{equation*}
    \Omega_\eta := \{(\zeta,x) \mid 0<\zeta <\eta, \; x\in\mathbb{R}\}.
\end{equation*}
Then, $\widehat{C}^A$ and $\widehat{P}^A$ satisfy the following obstacle problems, respectively:
\begin{equation} \label{CO}
    \begin{cases}
    {\widehat C^A(\zeta,x)\geq(e^x-K_c)^+\quad\text{in }\Omega_{T_c},}\\
        \partial_\zeta \widehat{C}^A(\zeta,x) -\mathcal{L} \widehat{C}^A(\zeta,x) \geq 0 \quad \text{for }\ (\zeta,x)\in \Omega_{T_c} \  \text{ with }\ \widehat{C}^A(\zeta,x) = (e^{x} - K_c)^+,\\
        \partial_\zeta \widehat{C}^A(\zeta,x) -\mathcal{L} \widehat{C}^A(\zeta,x) = 0 \quad \text{for }\ (\zeta,x)\in \Omega_{T_c} \  \text{ with }\   \widehat{C}^A(\zeta,x) > (e^{x} - K_c)^+,\\
        \widehat{C}^A(0,x) = (e^{x}-K_c)^+ \quad \text{for } \ x\in\mathbb{R},
    \end{cases}
\end{equation}
\begin{equation}\label{PO}
    \begin{cases}
    {\widehat P^A(\zeta,x)\geq(K_p-e^x)^+\quad\text{in }\Omega_{T_p},}\\
        \partial_\zeta \widehat{P}^A(\zeta,x) -\mathcal{L} \widehat{P}^A(\zeta,x) \geq 0 \quad  \text{for } \ (\zeta,x)\in \Omega_{T_p} \ \text{ with }\  \widehat{P}^A(\zeta,x) = (K_p - e^x)^+,\\
        \partial_\zeta \widehat{P}^A(\zeta,x) -\mathcal{L} \widehat{P}^A(\zeta,x) = 0 \quad \text{for } \ (\zeta,x)\in \Omega_{T_p} \ \text{ with }\   \widehat{P}^A(\zeta,x) > (K_p - e^x)^+,\\
        \widehat{P}^A(0,x) = (K_p - e^x)^+ \quad  \text{for } \ x\in\mathbb{R},
    \end{cases}
\end{equation}
where the differential operator $\mathcal{L}$ is given by 
 \begin{equation*} 
\mathcal{L}:=\frac{\sigma^2}{2}\partial_{xx}+(r-q-\frac{\sigma^2}{2})\partial_x-r.
\end{equation*}
According to the vast literature on American options (see, for instance, \cite{BX09,JJ19,JIA05,PAS11,YJB06,ZF09,ZF09-1}), the obstacle problems \eqref{CO} and \eqref{PO} have  strong solutions,
$\widehat{C}^A \in W^{1,2}_{p\textup{,}\, \text{loc}}(\Omega_{T_c}) \cap C(\overline{\Omega_{T_c}})$ and
$\widehat{P}^A \in W^{1,2}_{p\textup{,}\, \text{loc}}(\Omega_{T_p}) \cap C(\overline{\Omega_{T_p}})$, respectively{, and these are unique within the
class of solutions satisfying an exponential growth condition
\[
    |\widetilde U(\zeta,x)|\leq Ce^{\lambda|x|},
    \qquad C,\lambda>0,
\]
which is the growth class in which the comparison principle
\cite[Theorem 3]{Rapuch2005} is applicable. It is also well known that the standard American call and put values have
at most exponential growth in the log-price variable. Part~(i) of Lemma~\ref{lem:inequalityCP} below gives
a more precise estimate in the present setting.}

{
We finally rewrite \eqref{eqV*} in the same variables, setting
$\tau=T-t$ and $x=\log s$. Throughout Sections~\ref{sec:3}
and~\ref{sec:4}, the symbol $V$ denotes an unknown function on
$\Omega_T$ and \emph{not} a transform of the stopping value
\eqref{V_ch}: it is constructed in Section~\ref{sec:3} as a solution of
the forward non-degenerate obstacle problem
}
\begin{equation}\label{AC}
    \begin{cases}
    {V(\tau,x)\geq\max\{C(\tau,x),P(\tau,x)\}\quad\text{in }\Omega_T,}\\
        \partial_\tau V(\tau,x) -\mathcal{L} V(\tau,x) \ge 0 \quad \text{for } \ (\tau,x)\in \Omega_T \ \text{ with }\  V(\tau,x)=\max\{C(\tau,x),\;P(\tau,x)\},\\
        \partial_\tau V(\tau,x) -\mathcal{L} V(\tau,x) = 0 \quad \text{for } \ (\tau,x)\in \Omega_T \ \text{ with } \  V(\tau,x)>\max\{C(\tau,x),\;P(\tau,x)\},\\
        V(0,x) = \max\{C(0,x),P(0,x)\} \quad \text{for }\ x\in\mathbb{R},
    \end{cases}
\end{equation}
where 
\begin{equation}\label{eq:CA}
   \begin{cases} C(\tau,x):=\widehat{C}^A(T_c-T+\tau,x)=C^A(T-\tau,e^x),\\ P(\tau,x):=\widehat{P}^A(T_p-T+\tau,x)=P^A(T-\tau,e^x), 
   \end{cases} \quad (\tau,x)\in \overline{\Omega_T}.
\end{equation}
{
Only in Appendix~\ref{Appendix.B} is it shown that the inverse transform
$(t,s)\mapsto V(T-t,\log s)$ of the solution constructed in
Section~\ref{sec:3} coincides with the value $V^{\rm ch}$ defined by
\eqref{V_ch}. From that point on, and only then, the two notations may
be identified.
}

\subsection{Properties of American Call and Put Options}
\label{Prop: C,P}
The payoff function of the American chooser option takes the maximum form of the American call option $C^A$ and the American put option $P^A$. Therefore, the properties of $C^A$ and $P^A$ are essential for the analysis of the American chooser option. Since the properties of $C^A$ and $P^A$ naturally extend to $\widehat{C}^A$ and $\widehat{P}^A$, we briefly summarize the well-known properties of the American call and American put options in terms of $\widehat{C}^A$ and $\widehat{P}^A$ for convenience.

In each obstacle problem \eqref{CO} and \eqref{PO}, 
it is well known that contact with the respective obstacles occurs only at their positive parts.
Thus, we define the exercise regions, denoted by $\mathcal{E}_C$ and $\mathcal{E}_P$, and the continuation regions, denoted by $\mathcal{C}_C$ and $\mathcal{C}_P$, as follows:
\[
\mathcal{E}_C:=\{(\zeta,x)\in \Omega_{T_c} \, : \, \widehat{C}^A(\zeta,x)=e^x-K_c\},\quad 
\mathcal{C}_C:=\{(\zeta,x)\in \Omega_{T_c} \, :\, \widehat{C}^A(\zeta,x)>(e^x-K_c)^+\},
 \]
\[        
\mathcal{E}_P:=\{(\zeta,x)\in \Omega_{T_p} \, :\, \widehat{P}^A(\zeta,x)=K_p-e^x\},\quad
\mathcal{C}_P:=\{(\zeta,x)\in \Omega_{T_p} \, :\, \widehat{P}^A(\zeta,x)>(K_p-e^x)^+\}.
\]
Then, the free boundaries $\hat{x}_c(\tau)$ for the American call and $\hat{x}_p(\tau)$ for the American put are well-defined as follows:
\begin{equation*}
    \hat{x}_c(\tau):=\partial\mathcal{E}_C=\inf\{x\in\mathbb{R}\, :\, (\tau,x)\in \mathcal{E}_C\}\quad \text{and}\quad \hat{x}_p(\tau):=\partial\mathcal{E}_P=\sup\{x\in\mathbb{R}\, :\, (\tau,x)\in \mathcal{E}_P\}.
\end{equation*}
It is well known that $\hat{x}_c(\tau)$ and $\hat{x}_p(\tau)$ satisfy the following properties (see \cite{Blanchet2006, Chen2012, Detemple2005, Peskir2005, ZF09, ZF09-1}):
\begin{itemize}
    \item $\hat{x}_p(\tau)$ is a smooth and strictly decreasing function for $\tau \in (0,T_p]$.
    \item $\hat{x}_c(\tau)$ is a smooth and strictly increasing function for $\tau \in (0,T_c]$.
    \item As $\tau$ approaches 0, the limiting behaviors of $\hat{x}_c(\tau)$ and $\hat{x}_p(\tau)$ are given by 
    \begin{equation*}
        \lim_{\tau \to 0^+} \hat{x}_c(\tau) = \ln{\left(\max\left\{1,\tfrac{r}{q}\right\}K_c\right)} 
        \quad \text{and} \quad
        \lim_{\tau \to 0^+} \hat{x}_p(\tau) = \ln{\left(\min\left\{1,\tfrac{r}{q}\right\}K_p\right)},
    \end{equation*}
    where $q>0$. If $q=0$, there does not exist the free boundary $\hat{x}_c(\tau)$ and  $ \lim_{\tau \to 0^+} \hat{x}_p(\tau) = \ln{K_p}$.
\end{itemize}
Additionally, the following property can also be obtained by \cite{Qiu2018}.
\begin{itemize}
    \item  There exists a unique {$\bar{x}\in\mathbb{R}$} such that
    \begin{equation} \label{barx}
        \widehat{C}^A(T_c-T, \bar{x}) = \widehat{P}^A(T_p-T,\bar{x}),\quad\text{or equivalently}\quad C(0,\bar{x})=P(0,\bar{x}).
    \end{equation}
\end{itemize}

According to the properties of the free boundaries of the American call and put options described above, since $K_c > K_p$, the following holds, regardless of the relationship between $T_c$ and $T_p$:
\begin{equation*}
    \hat{x}_c(\tau) > \hat{x}_p(\tau),\quad\quad\forall\;\tau\ge 0.
\end{equation*}
In other words, $\mathcal{E}_C$ and $\mathcal{E}_P$ are always disjoint.
Figure \ref{fig:main} illustrates the behaviors of the free boundaries $\hat{x}_c(\tau)$ and $\hat{x}_p(\tau)$. 
 \begin{figure}[ht]
    \centering
    \begin{subfigure}{0.45\textwidth} 
        \includegraphics[width=\linewidth]{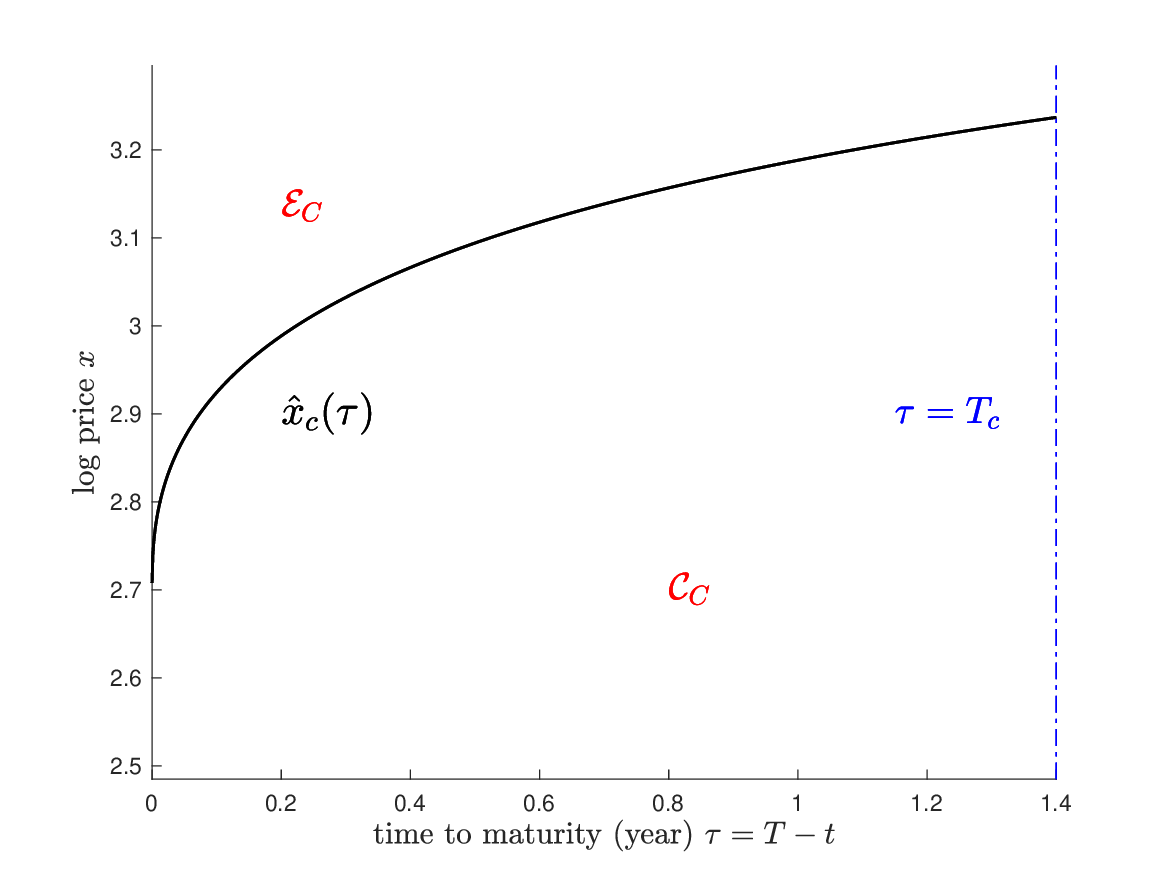}
        \caption{the free boundary $\hat{x}_c(\tau)$}
     \end{subfigure}
    \hfill
    \begin{subfigure}{0.45\textwidth}  
        \includegraphics[width=\linewidth]{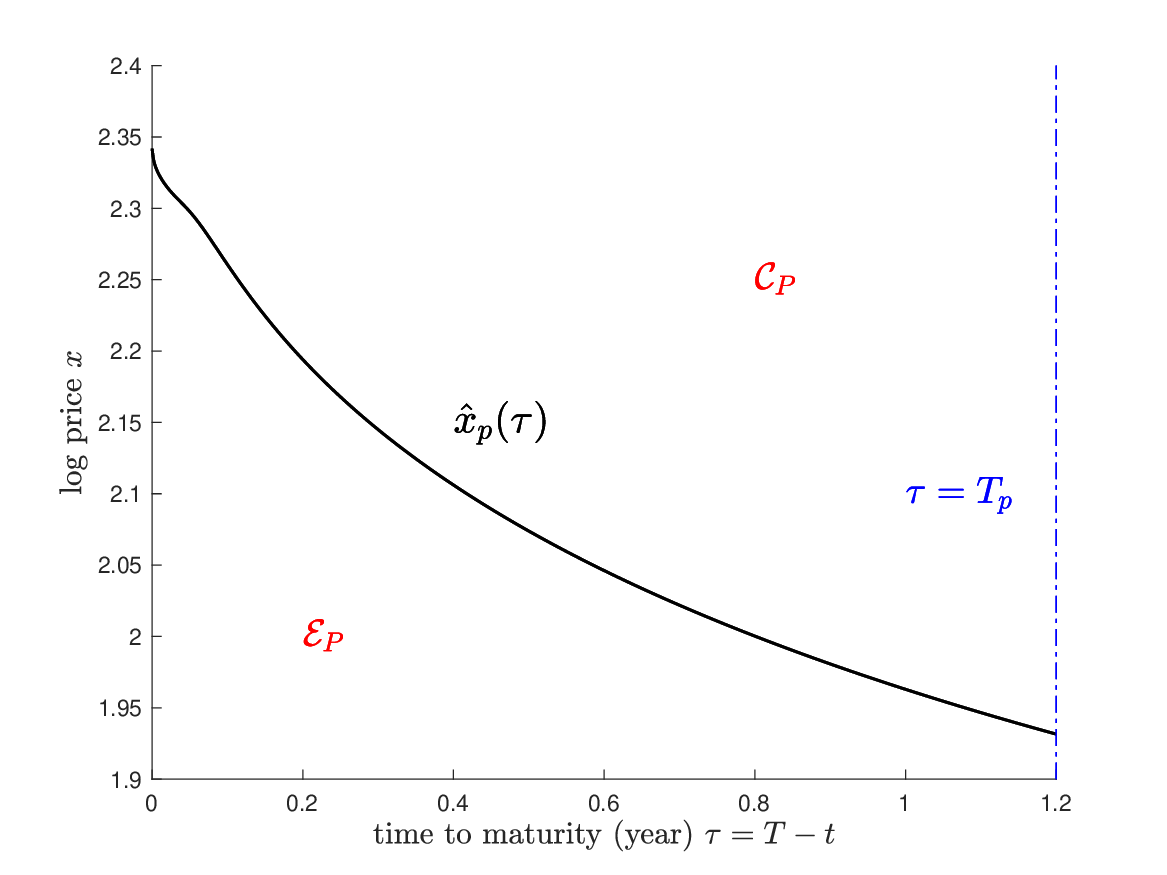}
                \caption{the free boundary $\hat{x}_p(\tau)$}
     \end{subfigure}
    \caption{The free boundaries $\hat{x}_c(\tau)$ and $\hat{x}_p(\tau)$. }
    \label{fig:main}
\end{figure}

\noindent In terms of $C$ and $P$, we also define $x_c(\tau)$ and $x_p(\tau)$ as follows:
\begin{equation*}
    x_c(\tau) := \hat{x}_c(T_c - T + \tau), \quad \text{and} \quad x_p(\tau) := \hat{x}_p(T_p - T + \tau).
\end{equation*}

For further discussion, it is necessary to  verify various properties of ${C}$ and ${P}$. 
Since ${C}$ and ${P}$ are restricted functions of $\widehat{C}^A$ and $\widehat{P}^A$, it is natural to analyze $\widehat{C}^A$ and $\widehat{P}^A$ over their original whole domains
and then conclude that the same properties hold for ${C}$ and ${P}$, respectively. The following results are essentially well-known and their proofs are provided in the Appendix. 
\begin{lemma}\label{lem:inequalityCP}
Let $\widehat{C}^A$ and $\widehat{P}^A$ be the {strong solutions to
\eqref{CO} and \eqref{PO}, respectively, satisfying an exponential
growth condition}, and define
$\Omega_{\lambda}:=(0,\lambda)\times\mathbb{R}$ for each $\lambda>0$.
Then the following properties hold.

\begin{enumerate}[label=(\roman*)]

\item
{
$0 \le \widehat{C}^A \le e^x$ in $\Omega_{T_c}$,
\quad
$0 \le \widehat{P}^A \le K_p$ in $\Omega_{T_p}$.
}

\item
$0\le \partial_x \widehat{C}^A \le e^x$ in $\Omega_{T_c}$,
\quad
$-e^{x}\le \partial_x \widehat{P}^A \le 0$ in $\Omega_{T_p}$.

\item
$\partial_\zeta \widehat{C}^A \ge 0$ in $\Omega_{T_c}$,
\quad
$\partial_\zeta \widehat{P}^A \ge 0$ in $\Omega_{T_p}$.

\item
{
For each $\tau\in(0,T]$, there exists a unique $x_\tau\in\mathbb R$
such that
\[
    C(\tau,x_\tau)=P(\tau,x_\tau),
\]
with
\[
    C(\tau,x)<P(\tau,x)\quad\text{for }x<x_\tau,
    \qquad
    C(\tau,x)>P(\tau,x)\quad\text{for }x>x_\tau.
\]
Moreover, $\tau\mapsto x_\tau$ is of class $C^1$ on $(0,T)$.
}

\end{enumerate}

Moreover, by \eqref{eq:CA}, the corresponding statements hold for
$C$ and $P$ in $\Omega_T$.
\end{lemma}



\section{Existence and uniqueness of solution}\label{sec:3}
Throughout, strong solutions lie above their obstacles everywhere and
satisfy the differential conditions almost everywhere, together with the
specified initial and boundary data; this also applies to the competitors
in the uniqueness statements.

In this section, we aim to prove the existence and uniqueness of the solution to the obstacle problem \eqref{AC}.

\subsection{Truncation and penalization}

{
Define
\[
    J(\tau,x):=\max\{C(\tau,x),P(\tau,x)\},
    \qquad (\tau,x)\in\Omega_T.
\]
Recall that $x_p(\tau)$ and $x_c(\tau)$ denote the free boundaries
of $P$ and $C$, respectively.
Fix $n\in\mathbb N$ sufficiently large so that
\[
    -n<x_p(\tau),
    \qquad
    x_c(\tau)<n
    \qquad\text{for all }\tau\in[0,T],
\]
and
\[
    2e^{-n}<K_p,
    \qquad
    e^n>K_c+K_p.
\]
Then, by Lemma~\ref{lem:inequalityCP} (i),
\[
    C(\tau,-n)\leq e^{-n}
    <K_p-e^{-n}
    =P(\tau,-n),
\]
and
\[
    P(\tau,n)\leq K_p
    <e^n-K_c
    =C(\tau,n)
\]
for all $\tau\in[0,T]$. Therefore,
\[
    J(\tau,-n)=K_p-e^{-n},
    \qquad
    J(\tau,n)=e^n-K_c
\]
for all $\tau\in[0,T]$.
}

We first consider the following obstacle problem in the bounded region
$\Omega_T^n:=(0,T)\times(-n,n)$
{with Dirichlet boundary conditions:}
\begin{align}\label{AC_R}
    \begin{cases}
    {V_n(\tau,x)\geq J(\tau,x)\quad\text{in }\Omega_T^n,}\\
    \partial_\tau V_n(\tau,x)-\mathcal{L}V_n(\tau,x)\geq 0 \quad 
   \text{for}\ \ (\tau,x)\in\Omega_T^n \ \ 
   \text{with} \ \ V_n(\tau,x)=J(\tau,x),  
   \\
   \partial_\tau V_n(\tau,x)-\mathcal{L}V_n(\tau,x)= 0 \quad 
   \text{for}\ \ (\tau,x)\in\Omega_T^n \ \ 
   \text{with} \ \ V_n(\tau,x)>J(\tau,x), 
   \\
   {
   V_n(\tau,-n)=K_p-e^{-n}
   \quad \text{and} \quad
   V_n(\tau,n)=e^n-K_c
   \quad \text{for }\tau\in[0,T],
   }
   \\
    V_n(0,x)=J(0,x) \quad \text{for }\ x\in (-n,n),
\end{cases}
\end{align}
We prove the existence and uniqueness of the solution to \eqref{AC_R} in Theorem~\ref{VN}. To this end, we use the so-called \textit{penalty method}.

{
Define a penalty function $\beta_{\varepsilon}(\cdot)\in C^{\infty}(\mathbb{R})$
with $\varepsilon>0$ satisfying
\begin{equation}\label{beta}
\begin{aligned}
&\beta_\varepsilon(\xi)\leq0,\qquad
\beta_\varepsilon'(\xi)\geq0,\qquad
\beta_\varepsilon''(\xi)\leq0
\quad\text{for all }\xi\in\mathbb R,\\
&\beta_\varepsilon(\xi)=0
\quad\text{for }\xi\geq\varepsilon,\qquad
\beta_\varepsilon(0)=-\mathcal K_0{(n)},\\
&\mathcal K_0{(n)}
:=2\big\{(q+r)e^n+2rK_c+r\big\},\\
&\lim_{\varepsilon\to0}\beta_\varepsilon(\xi)=0
\quad\text{for }\xi>0,\qquad
\lim_{\varepsilon\to0}\beta_\varepsilon(\xi)=-\infty
\quad\text{for }\xi<0.
\end{aligned}
\end{equation}
and $\varphi_\varepsilon(\cdot)\in C^\infty(\mathbb{R})$
satisfying \eqref{phiepsilon}.
}

{
Set
\begin{equation}\label{Jepsilon}
    J_{\varepsilon}:=\varphi_\varepsilon(C-P)+P.
\end{equation}
By the choice of $n$ above,
\[
    C(\tau,-n)<P(\tau,-n),
    \qquad
    P(\tau,n)<C(\tau,n)
\]
for all $\tau\in[0,T]$. Since $C$ and $P$ are continuous, there exists
$\varepsilon_n>0$ such that, for every
$0<\varepsilon<\varepsilon_n$,
\[
    C(\tau,-n)-P(\tau,-n)\leq-\varepsilon,
    \qquad
    C(\tau,n)-P(\tau,n)\geq\varepsilon
\]
for all $\tau\in[0,T]$.
Hence, by \eqref{phiepsilon},
\[
    J_\varepsilon(\tau,-n)=K_p-e^{-n}
    \quad\text{and}\quad
    J_\varepsilon(\tau,n)=e^n-K_c
    \qquad\text{for all }\tau\in[0,T].
\]
}

We then consider the following penalized problem:
\begin{equation}\label{AC_P}
    \begin{cases}
    \partial_{\tau} V_{n,\varepsilon}-\mathcal{L} V_{n,\varepsilon}
    +\beta_\varepsilon(V_{n,\varepsilon}-J_{\varepsilon})=0
    \quad \text{in }\ \Omega^{n}_T,
    \\
    {
    V_{n,\varepsilon}(\tau,-n)=K_p-e^{-n}
    \quad \text{and} \quad
    V_{n,\varepsilon}(\tau,n)=e^{n}-K_c
    \quad \text{for }\ \tau \in [0,T],
    }
    \\
    V_{n,\varepsilon}(0,x)=J_{\varepsilon}(0,x)
    \quad \text{for }\ x\in (-n,n).
\end{cases}
\end{equation}

\begin{theorem}\label{thm:existenceVnepsilon}
    {
    For each fixed $n$ chosen as above and each sufficiently small
    $\varepsilon>0$,
    }
    there exists a solution
    $V_{n,\varepsilon}\in W^{1,2}_p(\Omega_T^n) \cap
    C(\overline{\Omega_T^n})$ to the problem (\ref{AC_P}),
    where $1<p<\infty.$
\end{theorem}

\begin{proof} 
We apply the Schauder fixed point theorem \cite[280p]{DN01} to the following setting.
Set $\mathcal{X}:=C(\overline{\Omega_T^n})$ and
$\mathcal{A}:=\{w\in \mathcal{X}:w\geq 0 \}$.
Then $\mathcal{A}$ is a closed convex subset of the Banach space $\mathcal{X}.\ $

{
Fix $p>3$.
}
For each $w\in \mathcal{A}$, let $u\in W^{1,2}_p(\Omega_T^n)$
be the solution to
\begin{equation}\label{PDE}
    \begin{cases}
    \partial_{\tau} u-\mathcal{L} u+\beta_\varepsilon(w-J_{\varepsilon})=0
    \quad \text{in }\ \Omega^{n}_T,
    \\
    {
    u(\tau,-n)=K_p-e^{-n}
    \quad \text{and} \quad
    u(\tau,n)=e^{n}-K_c
    \quad \text{for }\ \tau \in [0,T],
    }
    \\
    u(0,x)=J_{\varepsilon}(0,x)
    \quad \text{for }\ x\in (-n,n).
\end{cases}
\end{equation}
{
By the choice of $n$ and $\varepsilon$, the compatibility conditions at
the corner points are satisfied:
\[
    J_\varepsilon(0,-n)=K_p-e^{-n},
    \qquad
    J_\varepsilon(0,n)=e^n-K_c.
\]
}
Note that the well-posedness of the solution to \eqref{PDE}
and the relevant estimate in $W^{1,2}_p(\Omega_T^n)$
can be found in \cite[Theorem 9.1, p.341]{LSU68}:
{
\[
\begin{aligned}
\Vert u \Vert_{W^{1,2}_p(\Omega_T^n)}
\leq \mathcal{K}\Big(
&\Vert K_p-e^{-n}\Vert_{W^1_p(0,T)}
+\Vert e^n-K_c\Vert_{W^1_p(0,T)}
\\
&+\Vert J_{\varepsilon}(0,\cdot)\Vert_{W^{2}_p(-n,n)}
+\Vert \beta_\varepsilon(w-J_{\varepsilon})
\Vert_{L^p(\Omega_T^n)}
\Big),
\end{aligned}
\]
}
for some constant $\mathcal{K}>0$.
Moreover, the parabolic Sobolev embedding theorem
\cite[Chapter II. Lemma 3.3]{LSU68} yields that $u$ is H\"older
continuous in $\overline{\Omega_T^n}$.
Thus $u\in\mathcal{X}.$
Now we define an operator $\mathcal{F}:\mathcal{A}\rightarrow\mathcal{X}$ by
$\mathcal{F}(w):=u$.
In order to use the Schauder fixed point theorem in $\mathcal{X},$
it suffices to show the following three properties:
\begin{itemize}
    \item[(1)] $\mathcal{F}(\mathcal{A})\subset \mathcal{A}$;
    \item[(2)] $\mathcal{F}$ is continuous;
    \item[(3)] $\mathcal{F}(\mathcal{A})$ is precompact in $\mathcal{X}$.
\end{itemize}
We shall prove it in Appendix \ref{Appendix.D}.
Utilizing the Schauder fixed point theorem, we obtain the solution
$V_{n,\varepsilon}$ of the problem (\ref{AC_P}).

{
Since
$\beta_\varepsilon(V_{n,\varepsilon}-J_\varepsilon)$ is bounded in $\Omega_T^n$
for each fixed $\varepsilon>0$, the same $W^{1,2}_p$-estimate applied for arbitrary $1<p<\infty$ yields
\[
    V_{n,\varepsilon}\in
    W^{1,2}_p(\Omega^n_T)\cap C(\overline{\Omega^n_T})
\]
for each $1<p<\infty$.
}
\end{proof}

Next, we investigate the convergence of $V_{n,\varepsilon}$ as
$\varepsilon \to 0^+$.
We establish that the limit exists and is a solution to the obstacle
problem \eqref{AC_R}.

\begin{lemma}\label{UBL}
{
Let $V_{n,\varepsilon}$ be the solution of \eqref{AC_P}.
For each sufficiently large $n$ and all sufficiently small
$\varepsilon\in(0,1)$,
\[
 -\mathcal K_0(n)\leq
 \beta_\varepsilon(V_{n,\varepsilon}-J_\varepsilon)\leq0,
 \qquad
 0\leq V_{n,\varepsilon}(\tau,x)\leq e^x+\mathcal K_1
 \quad\text{in }\Omega_T^n.
\]
Here $\mathcal K_0(n)$ is the penalty constant specified in \eqref{beta},
which may depend on $n$, whereas $\mathcal K_1\geq K_p+2$ can be chosen
independently of both $n$ and $\varepsilon$.
}
\end{lemma}

\begin{proof}

From Lemma~\ref{lem:inequalityCP} (i), we have
\[
\vert C-P\vert\leq C+P\leq { e^n+K_p
}
\quad\text{in}\ \Omega_T^n,
\]
and since $C,P\in W^{1,2}_{p,\text{loc}}(\Omega_T)$,
it follows that almost everywhere in $\Omega_T^n$,
\begin{align*}
    \partial_\tau C-\mathcal{L}C
    &=(qe^x-rK_c)\mathbf{1}_{\{C=e^x-K_c\}},
    \\
    \partial_\tau P-\mathcal{L}P
    &=(-qe^x+rK_p)\mathbf{1}_{\{P=K_p-e^x\}}.
\end{align*}
Direct computation and the choice of $\mathcal{K}_0{(n)}$ yield the following
in $\Omega_T^n$:
\begin{align*}
\partial_\tau J_\varepsilon - \mathcal{L} J_\varepsilon
&= \varphi'_\varepsilon(C-P)
   \left( \partial_\tau C - \mathcal{L}C \right)
 + \left( 1 - \varphi'_\varepsilon(C-P) \right)
   \left( \partial_\tau P - \mathcal{L}P \right) \\
&\quad - \frac{\sigma^2}{2}
\varphi''_\varepsilon(C-P)
\left[ \partial_x(C-P) \right]^2 
 - r\varphi'_\varepsilon(C-P)(C-P)
 + r\varphi_\varepsilon(C-P) \\
&\leq {
\varphi'_\varepsilon(\cdot)(qe^n+rK_c)
+\left(1-\varphi'_\varepsilon(\cdot)\right)(qe^n+rK_p)
} \\
&\quad +2r\left(\max_{\Omega_T^n}|C-P|+1\right) \\
&\leq {
qe^n+rK_c+2r(e^n+K_p+1)
} \\
&\leq {
2(q+r)e^n+4rK_c+2r
} \\
&={\mathcal{K}_0{(n)}}
=-\beta_\varepsilon(0).
\end{align*}
for sufficiently small $0<\varepsilon<1$.
Consequently, we deduce
\begin{equation*}
     \begin{cases}
    \partial_{\tau}J_{\varepsilon}-\mathcal{L} J_{\varepsilon}
    +\beta_\varepsilon(0)\leq0
    \quad \text{in }\  \Omega^{n}_T,
    \\
    {
    J_{\varepsilon}(\tau,-n)=K_p-e^{-n}
    =V_{n,\varepsilon}(\tau,-n)
    \quad\text{and}
    }
    \\
    {
    J_{\varepsilon}(\tau,n)=e^n-K_c
    =V_{n,\varepsilon}(\tau,n)
    \quad \text{for }\ \tau\in [0,T],
    }
    \\
    J_{\varepsilon}(0,x)=V_{n,\varepsilon}(0,x)
    \quad \text{for }\ x\in (-n,n).
    \end{cases}
\end{equation*}
Hence, $J_{\varepsilon}$ is a subsolution to (\ref{AC_P}) 
and by the comparison principle \cite[Theorem 1.1]{YY08}
\begin{equation*}
V_{n,\varepsilon}\geq J_{\varepsilon} \geq 0
\quad \text{in }\  \Omega_T^n.    
\end{equation*}
Moreover, the monotonicity of $\beta_\varepsilon$ implies
\[
-\mathcal{K}_0{(n)}=\beta_{\varepsilon}(0)
\leq \beta_\varepsilon(V_{n,\varepsilon}-J_{\varepsilon})
\leq 0.
\]

We next show that $V_{n,\varepsilon}$ is bounded from above.
{
To this end, define
\[
    \overline{V}(\tau,x):=e^x+\mathcal{K}_1,
\]
where $\mathcal{K}_1$ is a constant satisfying
\[
    \mathcal{K}_1\geq K_p+2.
\]
By \eqref{phiepsilon}, we have
\[
    J_\varepsilon
    =\varphi_\varepsilon(C-P)+P
    \leq (C-P+\varepsilon)^++P
    =\max\{C+\varepsilon,P\}.
\]
Hence, by Lemma~\ref{lem:inequalityCP} (i) and
$0<\varepsilon<1$,
\[
    J_\varepsilon+\varepsilon
    \leq
    \max\{C+2\varepsilon,P+\varepsilon\}
    \leq
    \max\{e^x+2,K_p+1\}
    \leq
    \overline V.
\]
Therefore,
\[
    \beta_\varepsilon(\overline V-J_\varepsilon)=0
    \quad\text{in }\Omega_T^n.
\]
Moreover,
\[
    \partial_\tau\overline V-\mathcal L\overline V
    =qe^x+r\mathcal K_1>0
    \quad\text{in }\Omega_T^n.
\]
On the lateral boundary,
\[
    \overline V(\tau,-n)
    =e^{-n}+\mathcal K_1
    \geq K_p-e^{-n}
    =V_{n,\varepsilon}(\tau,-n),
\]
and
\[
    \overline V(\tau,n)
    =e^n+\mathcal K_1
    \geq e^n-K_c
    =V_{n,\varepsilon}(\tau,n).
\]
Also,
\[
    \overline V(0,x)
    \geq J_\varepsilon(0,x)
    =V_{n,\varepsilon}(0,x)
    \quad\text{for }x\in(-n,n).
\]
}
Thus, $\overline V$ is a supersolution to \eqref{AC_P}.
By the comparison principle {\cite[Theorem 1.1]{YY08}},
\[
    V_{n,\varepsilon}\leq\overline V
    {= e^x+\mathcal K_1}
    \quad\text{in }\Omega_T^n.
\]
\end{proof}

We now establish the existence of the solution $V_n$ to problem \eqref{AC_R}. 

\begin{theorem} \label{VN}
   {
   For each fixed $n$ chosen as above,
   there exists a solution
   $V_n\in W^{1,2}_{p,\mathrm{loc}}(\Omega^{n}_T)$
   to the problem (\ref{AC_R}), where $1<p<\infty$.
   }
   Moreover, $V_n\in C(\overline{\Omega^{n}_T})$.
\end{theorem}

\begin{proof}
{
By Lemma \ref{UBL}, there exists a constant $\mathcal K>0$
independent of $\varepsilon$ such that
\[
    \Vert V_{n,\varepsilon}\Vert_{L^\infty(\Omega_T^n)}
    +
    \Vert\beta_\varepsilon
    (V_{n,\varepsilon}-J_\varepsilon)\Vert_{L^\infty(\Omega_T^n)}
    \leq \mathcal K.
\]
Moreover, from Lemma \ref{lem:inequalityCP} (ii),
$0\leq\varphi'_\varepsilon\leq1$, and
\[
    \partial_xJ_\varepsilon
    =
    \{1-\varphi'_\varepsilon(C-P)\}\partial_xP
    +
    \varphi'_\varepsilon(C-P)\partial_xC,
\]
we obtain
\[
    -e^x\leq\partial_xJ_\varepsilon\leq e^x
    \quad\text{in }\Omega_T^n.
\]
Hence,
\[
    [J_\varepsilon(0,\cdot)]_{C^{0,1}([-n,n])}\leq e^n
\]
uniformly in $\varepsilon$.
Together with the lateral boundary data and the corner compatibility
conditions, the boundary H\"older estimates \cite[Theorem 6.33]{LIE96} yield some
$\alpha\in(0,1)$ such that
\[
    \Vert V_{n,\varepsilon}\Vert_
    {C^{\frac{\alpha}{2},\alpha}
    (\overline{\Omega_T^n})}
    \leq \mathcal K',
\]
where $\mathcal K'$ is independent of $\varepsilon$.

Therefore, by the Arzel\`a--Ascoli theorem, there exist
$\varepsilon_j\to0^+$ and
$V_n\in C(\overline{\Omega_T^n})$ such that
\begin{align}
    V_{n,\varepsilon_j}\rightarrow V_n
    \quad &\text{in }
    C(\overline{\Omega_T^n}).
    \label{WC_2}
\end{align}

For each $Q\Subset\Omega_T^n$ and $1<p<\infty$,
the interior $W^{1,2}_p$-estimate gives
\[
    \Vert V_{n,\varepsilon_j}\Vert_{W^{1,2}_p(Q)}
    \leq
    \mathcal K_Q
    \left(
    \Vert V_{n,\varepsilon_j}\Vert_{L^p(\Omega_T^n)}
    +
    \Vert\beta_{\varepsilon_j}
    (V_{n,\varepsilon_j}-J_{\varepsilon_j})
    \Vert_{L^p(\Omega_T^n)}
    \right)
    \leq \mathcal K'_Q,
\]
where $\mathcal K'_Q$ is independent of $j$.
Hence, up to a subsequence,
\begin{align}
    V_{n,\varepsilon_j}\rightharpoonup V_n
    \quad &\text{in }W^{1,2}_p(Q),
    \label{WC_1}
\end{align}
where we have used \eqref{WC_2} to identify the weak limit with $V_n$.
Since $Q\Subset\Omega_T^n$ is arbitrary,
\[
    V_n\in W^{1,2}_{p,\mathrm{loc}}(\Omega_T^n)
\]
for every $1<p<\infty$.
}

Now, we prove that $V_n$ satisfies \eqref{AC_R}.
Observe from the inequality $V_{n,\varepsilon}\ge J_\varepsilon$,
the uniform convergence {\eqref{WC_2}} and
$J_\varepsilon\to J$ in $\Omega_T^n$, as $\varepsilon\to 0^+$, yields
$V_n\ge J$ in $\Omega_T^n$.

{
Also, by \eqref{WC_2} and the boundary conditions of
$V_{n,\varepsilon}$,
\[
    V_n(\tau,-n)=K_p-e^{-n}
    \quad\text{and}\quad
    V_n(\tau,n)=e^n-K_c
    \quad\text{for }\tau\in[0,T],
\]
and
\[
    V_n(0,x)=J(0,x)
    \quad\text{for }x\in(-n,n).
\]
}

Fix any 
$\psi \in C_c^\infty (\{V_n>J\}\cap\Omega^{n}_T)$.
Then, there exist small {$\delta>0$} such that 
\[
  \min\limits_{\text{supp}(\psi)}(V_{n}-J)>2\delta.
\]
Choose $\varepsilon_0>0$ small so that for every
{$\varepsilon_j\in (0,\varepsilon_0]$},
\[
 \vert (V_{n,\varepsilon_j}-J_{\varepsilon_j})-(V_n-J) \vert<\delta
\quad  \text{and} \quad
\varepsilon_j <\delta/2.
\]
Then, for such small $\varepsilon_j>0$, it follows that
$V_{n,\varepsilon_j}-J_{\varepsilon_j}>\delta>\varepsilon_j$,
and this implies
$\beta_{\varepsilon_j}(V_{n,\varepsilon_j}-J_{\varepsilon_j})=0$
in $\text{supp}(\psi)$.
Consequently,
\[
\iint_{\Omega^n_T}
\big(\partial_\tau V_{n,\varepsilon_j}
-\mathcal{L}V_{n,\varepsilon_j}\big)\psi\, d\tau\, dx=0.
\]
By the weak convergence \eqref{WC_1}, taking
$\varepsilon_j\rightarrow0^+$ yields
\[
\iint_{\Omega^n_T}
\big(\partial_\tau V_{n}-\mathcal{L}V_{n}\big)\psi\, d\tau\, dx=0.
\]
Since $\psi\in C_c^\infty (\{V_n>J\}\cap\Omega^{n}_T)$ is arbitrary,
we deduce that
$\partial_\tau V_{n}-\mathcal{L}V_{n}=0$ a.e. in 
$\{V_n>J\}\cap\Omega^{n}_T.$ 

To prove
$\partial_\tau V_n-\mathcal{L}V_n\ge 0$ a.e. in $\Omega^{n}_T$,
let us fix a non-negative $\psi\in C_c^\infty(\Omega^{n}_T)$.
Then since $\beta_{\varepsilon_j}(\cdot)\le 0$,
$V_{n,\varepsilon_j}$ satisfies that for all small $\varepsilon_j>0$,
\begin{align*}
    \iint_{\Omega^{n}_T}
    ({\partial_\tau V_{n,\varepsilon_j}
    -\mathcal{L}V_{n,\varepsilon_j}})\psi \, d\tau \, dx 
    =
    \iint_{\Omega^{n}_T}
    -\beta_{\varepsilon_j}
    (V_{n,\varepsilon_j}-J_{\varepsilon_j})
    \psi \, d\tau \, dx \geq0.
\end{align*}
By the weak convergence \eqref{WC_1}, taking
$\varepsilon_j\rightarrow0^+$ yields
\begin{align*}
    \iint_{\Omega^{n}_T}
    ({\partial_\tau V_{n}-\mathcal{L}V_{n}})\psi
    \, d\tau \, dx
    \geq0.
\end{align*}
Since $\psi$ is an arbitrary non-negative test function in
$C_c^\infty(\Omega_T^n)$, we deduce that
$\partial_\tau V_n-\mathcal{L}V_n\ge 0$ a.e. in $\Omega_T^n$.
Combining all, we conclude that $V_n$ is a solution of \eqref{AC_R}.
\end{proof}

\subsection{{Existence and uniqueness of the solution}}
Finally, we prove the existence and uniqueness of the solution to the obstacle
problem \eqref{AC}.
\begin{theorem}\label{thm:strong-solution}
    There exists {a solution} $V\in 
    W^{1,2}_{p, \mathrm{loc}}(\Omega_T)\cap C(\overline{\Omega_T})$
    to \eqref{AC}, where $1<p<\infty$. 
    {
    Moreover, $V$ satisfies
    \[
        |V(\tau,x)|\leq M e^{|x|}
        \quad\text{in }\overline{\Omega_T}
    \]
    for some $M>0$, and is unique among solutions satisfying an
    exponential growth condition
    \[
        |\widetilde V(\tau,x)|
        \leq C e^{\lambda|x|}
        \quad\text{in }\overline{\Omega_T}
    \]
    for some $C,\lambda>0$.
    }
\end{theorem}

\begin{proof}
For each $n\in\mathbb{N}$, 
let $V_n$ be the solution to the obstacle problem \eqref{AC_R}
constructed in the previous subsection.
By the {local} $W^{1,2}_p$-regularity of $V_n$, we have
\begin{equation*}
    \begin{cases}
    \partial_\tau V_n-\mathcal{L}V_n=f_n
    \quad \text{a.e. in }\ \Omega_T^n,\\
    {
    V_n(\tau,-n)=K_p-e^{-n} \quad \text{and} \quad
    V_n(\tau,n)=e^n-K_c
    \quad  \text{for }\ \tau \in [0,T],
    }\\
    V_n(0,x)={ J(0,x)} \quad \text{for }\ x\in(-n,n),
    \end{cases}
\end{equation*}
where 
\[
f_n(\tau,x)=
(qe^x-rK_c) \mathbf{1}_{\{V_n=e^x-K_c \}\cap \{C>P \}}
+(-qe^x+rK_p) \mathbf{1}_{\{V_n=K_p-e^x \}\cap \{P>C\}}.
\]
{
By Lemma~\ref{lem:inequalityCP}(iv), the set $\{C=P\}$ is the graph
of a $C^1$ function and hence has zero Lebesgue measure in
$\Omega_T^n$. Thus, the above identity holds almost everywhere.
}

Let us fix some $m\in\mathbb{N}$. Then for each $n>2m$,
$f_n$ is uniformly bounded in $\Omega_T^m$,
with a bound depending only on $m$ and independent of $n$.
{
Moreover, by taking $\varepsilon\to 0^+$ at the estimate obtained in the previous subsection, we deduce that
\[
    0\leq V_n(\tau,x)\leq e^x+\mathcal K_1
    \quad\text{in }\Omega_T^n,
\]
and hence $\{V_n\}_{n>2m}$ is uniformly bounded in
$\Omega_T^{2m}$.
Since $n>2m$, the lateral boundary of $\Omega_T^m$ is contained in
the interior of $\Omega_T^n$. Moreover,
$J(0,\cdot)=\max\{C(0,\cdot),P(0,\cdot)\}$ is Lipschitz continuous on
$[-2m,2m]$. Hence, applying the boundary H\"older estimate
\cite[Theorem~6.33]{LIE96} near $\{\tau=0\}\times[-m,m]$ and the
interior H\"older estimate in the remaining lateral part of
$\overline{\Omega_T^m}$, we obtain some $\alpha\in(0,1)$ such that
\[
    \Vert V_n\Vert_
    {C^{\frac{\alpha}{2},\alpha}(\overline{\Omega_T^m})}
    \leq \mathcal K_m,
    \qquad n>2m,
\]
where $\mathcal K_m$ is independent of $n$.
Therefore, $\{V_n\}_{n>2m}$ is precompact in
$C(\overline{\Omega_T^m})$.

In addition, for each $Q\Subset\Omega_T^m$, it follows from the
interior $W^{1,2}_p$-estimate for $1<p<\infty$ in
\cite[p.~355]{LSU68} that $\{V_n\}_{n>2m}$ is uniformly bounded in
$W^{1,2}_p(Q)$.
}

By a standard diagonal argument, 
we can find a subsequence $\{V_{n_k}\}_{k=1}^{\infty}$
{and functions $V^m$, $m\in\mathbb N$,}
such that, as $k\rightarrow\infty$,
{
\begin{align}\label{WC_3}
    V_{n_k}&\rightarrow V^m
    \quad \text{in }C(\overline{\Omega_T^m}), \notag\\
    V_{n_k}&\rightharpoonup V^m
    \quad \text{in }W^{1,2}_p(Q)
    \quad\text{for every }Q\Subset\Omega_T^m .
\end{align}
}
{
Since the same subsequence converges to $V^m$ and $V^{m+1}$
on $\overline{\Omega_T^m}$, we have
$V^{m+1}=V^m$ in $\overline{\Omega_T^m}$.
}
Thus, we can define
$V:=V^m$ in $\overline{\Omega_T^m}$ for all large $m\in \mathbb N$.
{
Moreover, $V\in W^{1,2}_{p,\mathrm{loc}}(\Omega_T)$
for every $1<p<\infty$.
}

We now proceed to prove that $V$ is a solution to \eqref{AC}. 
Let us fix $m \in \mathbb{N}$ and show that $V^m$ satisfies
\begin{equation*}
    \begin{cases}
    \partial_\tau V^m - \mathcal{L}V^m \geq 0 & \text{for } (\tau,x) \in \Omega^m_T \text{ with } V^m(\tau,x) = J(\tau,x), \\
    \partial_\tau V^m - \mathcal{L}V^m = 0    & \text{for } (\tau,x) \in \Omega^m_T \text{ with } V^m(\tau,x) > J(\tau,x), \\
    V^m(0,x) = J(0,x)                         & \text{for } x \in (-m,m).
    \end{cases}
\end{equation*}
Recall that for each $n \ge m$, $\partial_\tau V_{n} - \mathcal{L} V_{n} \geq 0$ a.e. in $\Omega_T^m$. 
{
Since $V_{n_k}\rightharpoonup V^m$ weakly in
$W^{1,2}_p(Q)$ for every $Q\Subset\Omega_T^m$,
}
we observe that for any non-negative test function $\psi \in C_c^\infty(\Omega_T^m)$,
\begin{equation*}
    \iint_{\Omega_T^m} (\partial_\tau V^m - \mathcal{L} V^m) \psi \, d\tau dx
    =
    {\lim_{k \to \infty}}
    \iint_{\Omega_T^m}
    (\partial_\tau {V_{n_k}}-\mathcal{L}{V_{n_k}})
    \psi \, d\tau dx
    \ge 0.
\end{equation*}
Since $\psi \ge 0$ is arbitrary, it follows that $\partial_\tau V^m - \mathcal{L} V^m \ge 0$ a.e. in $\Omega_T^m$. 
{
Furthermore, by the uniform convergence in \eqref{WC_3},
$V_{n_k}\to V^m$ in $C(\overline{\Omega_T^m})$.
}
Since $V_n \ge J$ in $\Omega_T^m$, passing to the limit yields
\[ V^m \geq J \quad \text{in } \Omega^m_T. \]
{
The same uniform convergence and
$V_{n_k}(0,x)=J(0,x)$ imply
$V^m(0,x)=J(0,x)$ for $x\in(-m,m)$.
}
Consequently, it remains to show that $\partial_\tau V^m - \mathcal{L} V^m = 0$ on the set $\{V^m > J\}$.

Fix any $\psi\in C_c^\infty(\{V^m>J\}\cap\Omega^m_T)$.
Then, there exists $\delta>0$ such that
\[\min\limits_{\text{supp}(\psi)}(V^m-J)>\delta.\]
{
By the uniform convergence in \eqref{WC_3}, for all sufficiently
large $k$,
\[
    V_{n_k}-J>\frac{\delta}{2}
    \quad\text{on }\operatorname{supp}(\psi).
\]
Hence
$\operatorname{supp}(\psi)\subset\{V_{n_k}>J\}$, and therefore
\[
   \iint_{\Omega^m_T}
   (\partial_\tau V_{n_k}-\mathcal{L}V_{n_k})
   \psi \, d\tau dx=0.
\]
Passing to the limit by \eqref{WC_3}, we obtain
\[
   \iint_{\Omega^m_T}
   (\partial_\tau V^m-\mathcal{L}V^m)
   \psi\, d\tau dx=0.
\]
}
Thus, $\partial_\tau V^m -\mathcal{L}V^m=0$ a.e. in $\{V^m>J\}$. 
Since $m\in\mathbb{N}$ is arbitrary, we conclude that $V$ is a solution to \eqref{AC}.

{
For the uniqueness in the stated growth class, observe that the
estimate obtained above passes to the limit and gives
\[
    0\leq V(\tau,x)\leq e^x+\mathcal K_1
    \quad\text{in }\Omega_T.
\]
Thus $|V(\tau,x)|\leq M e^{|x|}$ for some $M>0$.
Moreover, by Lemma \ref{lem:inequalityCP} (i), the obstacle $J$
also satisfies an exponential growth bound.

Let $\widetilde V$ be another solution to \eqref{AC} satisfying
$|\widetilde V(\tau,x)|\leq C e^{\lambda|x|}$ for some
$C,\lambda>0$.
{
The comparison principle \cite[Theorem 1.2]{YY08} for variational
inequalities in unbounded sets is applicable to any pair of solutions
obeying a bound of this form, with no restriction on $\lambda>0$;
this is exactly the class in which uniqueness is asserted here, and it
contains $V$ by the preceding estimate.
Consequently that comparison principle gives}
$V\leq\widetilde V$ and $\widetilde V\leq V$ in $\Omega_T$.
Hence $\widetilde V\equiv V$, which proves uniqueness in the stated
class.
}
\label{rev:construction-end}
\end{proof}

{
\subsection{Estimate for the derivatives of \textit{V}}

\begin{lemma}\label{DV}
{Let $V$ be the solution of Theorem~\ref{thm:strong-solution}.} Then
\[
    \partial_\tau V\geq0
    \quad \text{a.e. in }\ \Omega_T
    \quad\text{and}\quad
    -e^x\leq\partial_xV\leq e^x
    \quad\text{ in }\ \Omega_T.
\]
\end{lemma}

\begin{proof}
To prove the first inequality, fix some small $\delta>0$.
The functions $V(\tau+\delta,x)$ and $V(\tau,x)$ satisfy, respectively,
\[
\begin{cases}
\partial_\tau V(\tau+\delta,x)-\mathcal LV(\tau+\delta,x)\geq0,
&\text{if }V(\tau+\delta,x)=J(\tau+\delta,x),\\
\partial_\tau V(\tau+\delta,x)-\mathcal LV(\tau+\delta,x)=0,
&\text{if }V(\tau+\delta,x)>J(\tau+\delta,x),\\
V(\delta,x),&\text{if }\tau=0,
\end{cases}
\]
and
\[
\begin{cases}
\partial_\tau V-\mathcal LV\geq0,
&\text{if }V=J(\tau,x),\\
\partial_\tau V-\mathcal LV=0,
&\text{if }V>J(\tau,x),\\
V(0,x)=J(0,x),&\text{if }\tau=0,
\end{cases}
\]
in $\Omega_{T-\delta}$.
By Lemma~\ref{lem:inequalityCP} (iii),
\[
    J(\tau+\delta,x)\geq J(\tau,x),
    \qquad
    V(\delta,x)\geq J(\delta,x)\geq J(0,x).
\]
Thus, the obstacles and the initial data are ordered in the same
direction, while the right-hand sides are both zero.
Since $V$ satisfies the exponential growth condition, the comparison principle for variational inequalities \cite[Theorem 1.2]{YY08} yields
\[
    V(\tau+\delta,x)\geq V(\tau,x)
    \quad\text{in }\Omega_{T-\delta}.
\]
Since $\delta>0$ is arbitrary,
$\partial_\tau V\geq0$ a.e. in $\Omega_T$.

Next, we prove the estimates for $\partial_xV$.
Fix $h>0$. By Lemma~\ref{lem:inequalityCP} (ii),
\[
    0\leq C(\tau,x+h)-C(\tau,x)\leq e^{x+h}-e^x,
\]
and
\[
    -(e^{x+h}-e^x)
    \leq P(\tau,x+h)-P(\tau,x)\leq0.
\]
Since $J=\max\{C,P\}$, it follows that
\begin{equation}\label{J_difference}
    |J(\tau,x+h)-J(\tau,x)|
    \leq e^{x+h}-e^x.
\end{equation}

We first prove the upper bound.
The functions $V(\tau,x+h)$ and
$V(\tau,x)+e^{x+h}-e^x$ satisfy, respectively,
\[
\begin{cases}
\partial_\tau V(\tau,x+h)-\mathcal LV(\tau,x+h)\geq0,
&\text{if }V(\tau,x+h)=J(\tau,x+h),\\
\partial_\tau V(\tau,x+h)-\mathcal LV(\tau,x+h)=0,
&\text{if }V(\tau,x+h)>J(\tau,x+h),\\
V(0,x+h)=J(0,x+h),&\text{if }\tau=0,
\end{cases}
\]
and, since $\mathcal Le^x=-qe^x$,
\[
\begin{cases}
\partial_\tau\!\left(V+e^{x+h}-e^x\right)
-\mathcal L\!\left(V+e^{x+h}-e^x\right)
\geq q(e^{x+h}-e^x),
\\
\hfill
\text{if }V+e^{x+h}-e^x
=J(\tau,x)+e^{x+h}-e^x,
\\
\partial_\tau\!\left(V+e^{x+h}-e^x\right)
-\mathcal L\!\left(V+e^{x+h}-e^x\right)
=q(e^{x+h}-e^x),
\\
\hfill
\text{if }V+e^{x+h}-e^x
>J(\tau,x)+e^{x+h}-e^x,
\\
V(0,x)+e^{x+h}-e^x
=J(0,x)+e^{x+h}-e^x,
\qquad \text{if }\tau=0.
\end{cases}
\]
By \eqref{J_difference},
\[
    J(\tau,x+h)
    \leq J(\tau,x)+e^{x+h}-e^x,
\]
and the same inequality holds for the initial condition.
Also, $0\leq q(e^{x+h}-e^x)$.
Thus, the obstacles, the right-hand sides, and the initial data
are ordered, and the comparison principle for variational inequalities \cite[Theorem 1.2]{YY08}
gives
\[
    V(\tau,x+h)
    \leq V(\tau,x)+e^{x+h}-e^x.
\]

For the lower bound, we compare $V(\tau,x)$ with
$V(\tau,x+h)+e^{x+h}-e^x$.
The latter satisfies the variational inequality with obstacle
$J(\tau,x+h)+e^{x+h}-e^x$, right-hand side
$q(e^{x+h}-e^x)$, and initial value
$J(0,x+h)+e^{x+h}-e^x$.
By \eqref{J_difference},
\[
    J(\tau,x)
    \leq J(\tau,x+h)+e^{x+h}-e^x,
\]
and the same ordering holds for the initial data, while
$0\leq q(e^{x+h}-e^x)$.
Hence, by the comparison principle for variational inequalities,
\[
    V(\tau,x)
    \leq V(\tau,x+h)+e^{x+h}-e^x.
\]
Combining the two estimates,
\[
    -(e^{x+h}-e^x)
    \leq V(\tau,x+h)-V(\tau,x)
    \leq e^{x+h}-e^x.
\]
Dividing by $h>0$ and letting $h\to0^+$ yields
\[
    -e^x\leq\partial_xV\leq e^x
    \quad\text{a.e. in }\Omega_T.
\]
Since $V\in W^{1,2}_{p,\mathrm{loc}}(\Omega_T)$ for every
$1<p<\infty$, the Sobolev embedding theorem implies that
$\partial_xV$ is continuous in $\Omega_T$. Therefore, the above
inequality holds everywhere in $\Omega_T$, and hence
\[
    -e^x\leq\partial_xV(\tau,x)\leq e^x,
    \qquad (\tau,x)\in\Omega_T.
\]
\end{proof}
}

\section{Analysis of the free boundary }\label{sec:4}
\subsection{{ Parametrization} of the free boundary}

Let us define the exercise region $\mathcal{E}$ and the continuation region $\mathcal{C}$ for the solution $V$ of  \eqref{AC} as
    \begin{align*}
    \mathcal{E}:=& \{ (\tau,x)\in \Omega_T: V(\tau,x)=J(\tau,x) \}\quad\text{and}\quad
    \mathcal{C}:=\{ (\tau,x)\in \Omega_T: V(\tau,x)>J(\tau,x) \}.
\end{align*}
Furthermore, let us define $\mathcal{E}_P^{\rm ch}$ and $\mathcal{E}_C^{\rm ch}$ as
\begin{align*}
    \mathcal{E}_P^{\rm ch}:=&\{ (\tau,x)\in \Omega_T:V(\tau,x)=P(\tau,x) \}
    \quad\text{and}\quad
    \mathcal{E}_C^{\rm ch}:=&\{ (\tau,x)\in \Omega_T: V(\tau,x)=C(\tau,x) \}.
\end{align*}

In the following lemma, we prove that $V$ contacts the obstacle exclusively along its smooth portions; that is, $V$ coincides with either $e^x-K_c$ or $K_p-e^x$ in the exercise region $\mathcal{E}$. 
This ensures that standard arguments remain applicable within our framework.

Furthermore, we establish that $\mathcal{E}_P^{\rm ch}$ and $\mathcal{E}_C^{\rm ch}$ are mutually disjoint. 
Consequently, the free boundary $\partial\mathcal{E}$ decomposes into the disjoint union of $\partial\mathcal{E}_P^{\rm ch}$ and $\partial\mathcal{E}_C^{\rm ch}$. 
Therefore, it suffices to analyze $\partial\mathcal{E}_P^{\rm ch}$ and $\partial\mathcal{E}_C^{\rm ch}$ separately.

\begin{lemma}\label{FB1}
Let $\mathcal{E}_P^{\rm ch}$ and $\mathcal{E}_C^{\rm ch}$ be defined as above. Then,
\begin{equation}\label{FB11}
\mathcal{E}_P^{\rm ch}\subset \{P=K_p-e^x\}
\quad \text{and} \quad
\mathcal{E}_C^{\rm ch}\subset \{C=e^x-K_c\}.
\end{equation}
In particular, we have
\begin{equation}\label{FB12}
\mathcal{E}_P^{\rm ch}=\{V=K_p-e^x\}
\quad \text{and} \quad
\mathcal{E}_C^{\rm ch}=\{V=e^x-K_c\},
\end{equation}
which further implies that $\mathcal{E}_P^{\rm ch}$ and $\mathcal{E}_C^{\rm ch}$ are mutually disjoint.
\end{lemma}

\begin{proof}

{
Suppose that
$(\tau_0,x_0)\in\{P>K_p-e^x\}\cap\mathcal E_P^{\rm ch}$.
By the continuity of $P$,
choose $\delta>0$ and $R>\max\{x_0,x_{\tau_0}\}$ such that
\[
    Q_R:=(\tau_0-\delta,\tau_0+\delta)\times(x_0-\delta,R)
    \Subset\Omega_T,
    \qquad
    Q_R\subset\{P>K_p-e^x\}.
\]
Then $w:=V-P\ge0$ satisfies
\[
    (\partial_\tau-\mathcal L)w\ge0
    \qquad\text{a.e. in }Q_R,
\]
belongs to $W_{p,\mathrm{loc}}^{1,2}(Q_R)$ for $p>3$, and
$w(\tau_0,x_0)=0$.
By \cite[Proposition~2.11]{CRAN00}, $w$ is an
$L^p$-viscosity supersolution. Since $w$ attains its minimum $0$
at the interior point $(\tau_0,x_0)$, the strong minimum principle
\cite[Corollary~2.4]{FRA04} yields
\[
    w=0
    \quad\text{on}\quad
    (\tau_0-\delta,\tau_0]\times(x_0-\delta,R).
\]
Choosing
$x\in(\max\{x_0,x_{\tau_0}\},R)$ then gives
$V(\tau_0,x)=P(\tau_0,x)$, contradicting
\[
    V(\tau_0,x)\ge C(\tau_0,x)>P(\tau_0,x)
\]
by Lemma~\ref{lem:inequalityCP}(iv).
Hence
\[
    \mathcal E_P^{\rm ch}\subset\{P=K_p-e^x\}.
\]
The call inclusion follows similarly in a bounded cylinder lying to the
left of the underlying call boundary, with left endpoint below
$x_{\tau_0}$, where $P>C$.}

Since $V \ge P \ge K_p - e^x$ and $V \ge C \ge e^x - K_c$, we have
\[
    \{V = K_p - e^x\} \subset \mathcal{E}_P^{\rm ch}
    \quad \text{and} \quad
    \{V = e^x - K_c\} \subset \mathcal{E}_C^{\rm ch}.
\]
Combined with \eqref{FB11}, these inclusions yield \eqref{FB12}.
Furthermore, since $\mathcal{E}_P$ and $\mathcal{E}_C$ are disjoint
(see Subsection~\ref{Prop: C,P}), the inclusions
$\mathcal{E}_P^{\rm ch} \subset \mathcal{E}_P$ and
$\mathcal{E}_C^{\rm ch} \subset \mathcal{E}_C$ immediately imply that
$\mathcal{E}_P^{\rm ch}$ and $\mathcal{E}_C^{\rm ch}$ are disjoint.
\label{rev:inclusion-end}
\end{proof}

{
We next show that both contact sets $\mathcal{E}_P^{\rm ch}$ and $\mathcal{E}_C^{\rm ch}$ have nonempty sections at every
time.

\begin{lemma}\label{FB: contact}
There exists some $h>0$ such that, for every $\tau\in(0,T)$,
\[
    V(\tau,x)=K_p-e^x
    \quad\text{for }x\leq-\frac{2}{h^2},
    \qquad
    V(\tau,x)=e^x-K_c
    \quad\text{for }x\geq\frac{2}{h^2}.
\]
In particular, both $\mathcal E_P^{\rm ch}$ and
$\mathcal E_C^{\rm ch}$ have nonempty sections at every
$\tau\in(0,T)$.
\end{lemma}

\begin{proof}
Set $b:=r-q-\frac{\sigma^2}{2}$. Choose $h>0$ sufficiently small so
that
\[
\begin{gathered}
    -\frac1{h^2}<\min\{x_p(0),\bar x\},
    \qquad
    \frac1{h^2}>\max\{x_c(0),\bar x\},\\
    q e^{-1/h^2}+h^3\sigma^2+2h|b|<rK_p,
    \qquad
    q e^{1/h^2}-rK_c-h^3\sigma^2-2h|b|>0,
\end{gathered}
\]
and
\[
    K_p-e^x>e^x-K_c
    \quad\text{for }x\leq-\frac1{h^2},
    \qquad
    e^x-K_c>K_p-e^x
    \quad\text{for }x\geq\frac1{h^2}.
\]
We further require
\[
    \frac1h\geq\mathcal K_1+K_c,
\]
where $V\leq e^x+\mathcal K_1$ in $\Omega_T$.

We first consider the left side. Set
\[
    D_-:=(0,T)\times
    \left(-\frac3{h^2},-\frac1{h^2}\right)
\]
and define
\[
    \underline z(x)
    :=
    K_p-e^x
    +h^3\left(x+\frac2{h^2}\right)^2.
\]
Since $-1/h^2<\min\{x_p(0),\bar x\}$,
\[
    V(0,x)=K_p-e^x\leq\underline z(x)
    \quad\text{on } \overline{D_-}\cap\{\tau=0\}.
\]
At both lateral boundaries of $D_-$, the quadratic term equals
$1/h$. Hence, by the choice of $h$ and the bound
$V\leq e^x+\mathcal K_1$,
\[
    V\leq\underline z
    \quad\text{on }\partial_pD_-.
\]
Moreover,
\[
\begin{aligned}
    -\mathcal L\underline z
    &=
    rK_p-qe^x
    -h^3\left[
        \sigma^2
        +2b\left(x+\frac2{h^2}\right)
        -r\left(x+\frac2{h^2}\right)^2
    \right] \\
    &\geq
    rK_p-qe^{-1/h^2}-h^3\sigma^2-2h|b|
    >0
    \qquad\text{in }D_-.
\end{aligned}
\]

Set $w:=V-\underline z$. Suppose to the contrary that
\[
    \mathcal O_-:=\{w>0\}\cap D_-
\]
is nonempty. Since
\[
    V>\underline z\geq K_p-e^x>e^x-K_c
    \quad\text{in }\mathcal O_-,
\]
Lemma~\ref{FB1} gives $\mathcal O_-\subset\mathcal C$. Hence,
\[
    (\partial_\tau-\mathcal L)w
    =
    \mathcal L\underline z<0
    \quad\text{a.e. in }\mathcal O_-.
\]
Also, $w\leq0$ on $\partial_p\mathcal O_-$ by the boundary
inequalities above and the definition of $\mathcal O_-$. The
comparison principle \cite[Corollary~7.4]{LIE96} therefore gives
$w\leq0$ in $\mathcal O_-$, a contradiction. Thus,
\[
    V\leq\underline z
    \quad\text{in }D_-.
\]
At $x=-2/h^2$, the quadratic term vanishes, and therefore
\[
    V\!\left(\tau,-\frac2{h^2}\right)
    =
    K_p-e^{-2/h^2}.
\]
Since Lemma~\ref{DV} implies that
$V(\tau,x)-(K_p-e^x)$ is nondecreasing in $x$, we conclude that
\[
    V(\tau,x)=K_p-e^x
    \quad\text{for }x\leq-\frac2{h^2}.
\]

We next consider the right side. Set
\[
    D_+:=(0,T)\times
    \left(\frac1{h^2},\frac3{h^2}\right)
\]
and define
\[
    \overline z(x)
    :=
    e^x-K_c
    +h^3\left(x-\frac2{h^2}\right)^2.
\]
As above,
\[
    V\leq\overline z
    \quad\text{on }\partial_pD_+,
\]
and
\[
\begin{aligned}
    -\mathcal L\overline z
    &=
    qe^x-rK_c
    -h^3\left[
        \sigma^2
        +2b\left(x-\frac2{h^2}\right)
        -r\left(x-\frac2{h^2}\right)^2
    \right] \\
    &\geq
    qe^{1/h^2}-rK_c-h^3\sigma^2-2h|b|
    >0
    \qquad\text{in }D_+.
\end{aligned}
\]
Set $w:=V-\overline z$. If
\[
    \mathcal O_+:=\{w>0\}\cap D_+
\]
were nonempty, then
\[
    V>\overline z\geq e^x-K_c>K_p-e^x
    \quad\text{in }\mathcal O_+,
\]
so Lemma~\ref{FB1} gives $\mathcal O_+\subset\mathcal C$. Thus,
\[
    (\partial_\tau-\mathcal L)w
    =
    \mathcal L\overline z<0
    \quad\text{a.e. in }\mathcal O_+,
\]
while $w\leq0$ on $\partial_p\mathcal O_+$. The comparison principle
again gives a contradiction. Therefore,
\[
    V\leq\overline z
    \quad\text{in }D_+.
\]
Since the quadratic term vanishes at $x=2/h^2$,
\[
    V\!\left(\tau,\frac2{h^2}\right)
    =
    e^{2/h^2}-K_c.
\]
Finally, Lemma~\ref{DV} implies that
$V(\tau,x)-(e^x-K_c)$ is nonincreasing in $x$, and hence
\[
    V(\tau,x)=e^x-K_c
    \quad\text{for }x\geq\frac2{h^2}.
\]
\end{proof}
}

To proceed further, recall from Subsection~\ref{Prop: C,P} that the $x$-coordinates of points in
$\mathcal{E}_C$ and $\mathcal{E}_P$ satisfy $x>\ln K_c$ and $x<\ln K_p$, respectively.
Combined with Lemma~\ref{FB1}, the same bounds hold for points in
$\mathcal{E}_C^{\rm ch}$ and $\mathcal{E}_P^{\rm ch}$, respectively.
Therefore, using Lemma~\ref{DV} and Lemma~\ref{FB1}, we parameterize the free boundaries by
\begin{align}\label{FB2-1}
     x_p^{\rm ch}(\tau)&:=\sup\limits_{x<\ln K_p} \ \{x: V(\tau,x)=K_p-e^x \}\quad \tau\in(0,T),\\
     x_c^{\rm ch}(\tau)&:=\inf\limits_{x>\ln K_c} \ \{x: V(\tau,x)=e^x-K_c \}\quad \tau\in(0,T).\label{FB2-2}
\end{align}
\begin{remark}
    {
    By Lemma~\ref{FB: contact}, the sets
    $\{x:V(\tau,x)=K_p-e^x\}$ and
    $\{x:V(\tau,x)=e^x-K_c\}$ are nonempty for every
    $\tau\in(0,T)$. Hence, $x_p^{\rm ch}(\tau)$ and
    $x_c^{\rm ch}(\tau)$ are well-defined and finite for every
    $\tau\in(0,T)$.
    }
    Moreover, the monotonicity of $V$ with respect to $\tau$, obtained in Lemma \ref{DV}, implies that $x_p^{\rm ch}(\tau)$ and
    $x_c^{\rm ch}(\tau)$ are monotone decreasing and increasing respectively.
\end{remark}

{
The two free boundaries of the chooser are located relative to those of
the underlying options as follows:
\begin{equation}\label{inq:frees}
    x_p^{\rm ch}(\tau) \le x_p(\tau) < x_c(\tau) \le x_c^{\rm ch}(\tau), \quad\quad \forall\; \tau \in [0, T).
\end{equation}
Indeed, Lemma~\ref{FB1} gives
$\mathcal E_P^{\rm ch}\subset\mathcal E_P$ and
$\mathcal E_C^{\rm ch}\subset\mathcal E_C$. Taking sections at a fixed
$\tau$ and using \eqref{FB2-1}, \eqref{FB2-2} together with the
definitions of $x_p$ and $x_c$ in Subsection~\ref{Prop: C,P}, the
supremum defining $x_p^{\rm ch}(\tau)$ is taken over a subset of the
set defining $x_p(\tau)$, whence
$x_p^{\rm ch}(\tau)\le x_p(\tau)$; similarly the infimum defining
$x_c^{\rm ch}(\tau)$ is taken over a subset of the set defining
$x_c(\tau)$, whence $x_c(\tau)\le x_c^{\rm ch}(\tau)$. The middle
inequality was recorded in Subsection~\ref{Prop: C,P}. The endpoint
$\tau=0$ is included by Theorems~\ref{thm:free-p}(ii) and
\ref{thm:free-c}(ii).
}

\begin{figure}[htbp]
    \centering
    \begin{subfigure}{0.6\textwidth}
        \includegraphics[width=\linewidth]{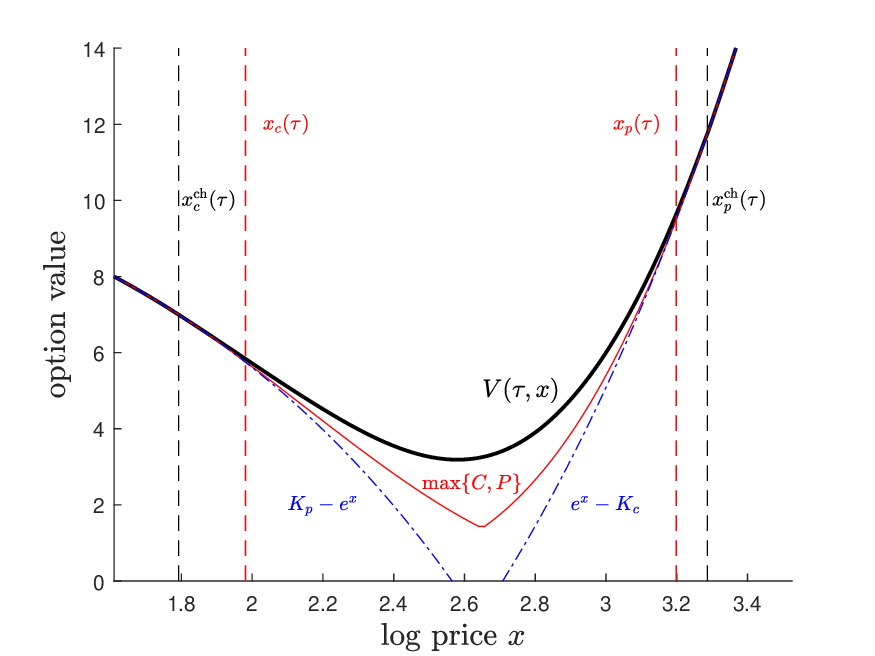}
     \end{subfigure}
    \caption{{The s}olution $V(\tau,x)$ and {the }obstacle $\max\{C(\tau,x),P(\tau,x)\}.$}
    \label{fig:main1}
\end{figure}

Figure \ref{fig:main1} illustrates the solution $V(\tau, x)$ of \eqref{AC} and the obstacle $J(\tau,x)=\max\{C(\tau, x), P(\tau, x)\}$,
{together with \eqref{inq:frees} and the fact, proven in
Lemma~\ref{FB1}, that $V(\tau,x)$ and the obstacle meet only in the
regions where $C(\tau, x) = e^x - K_c$ and $P(\tau, x) = K_p - e^x$.}

\subsection{Properties of the free boundary}
{
We first investigate the strict monotonicity and limiting behavior of
$x_p^{\rm ch}$ and $x_c^{\rm ch}$ defined in
\eqref{FB2-1} and \eqref{FB2-2}, respectively.
}

\begin{theorem}\label{thm:free-p}
Let $x_p^{\rm ch}(\tau)$ be defined as in \eqref{FB2-1}. Then,
\begin{itemize}
\item[i)] $x_p^{\rm ch}$ is strictly decreasing on $(0,T)$.
\item[ii)] $\displaystyle{\lim_{\tau \to 0^+}x_p^{\rm ch}(\tau)}
=\min\{\bar{x},{x_p}(0)\}$,
where $\bar{x}$ is the point stated in \eqref{barx}.
\end{itemize}
\end{theorem}

\begin{proof}[Proof of i)]
From Lemma \ref{DV}, we can easily verify that $x_p^{\rm ch}(\tau)$ is monotone decreasing. 
Assume, for contradiction, that $x_p^{\rm ch}(\tau)$ is not strictly decreasing.
Then there exists some $\tau_1,\tau_2>0$ with $\tau_1<\tau_2$ such that $x_p^{\rm ch}(\tau_1)=x_p^{\rm ch}(\tau_2)=:x_0.$

Monotonicity gives $x_p^{\rm ch}=x_0$ on $[\tau_1,\tau_2]$,
and \eqref{inq:frees} gives $x_c^{\rm ch}(\tau)\geq x_c(\tau)>\ln K_c$.
For $0<\delta<(\tau_2-\tau_1)/2$, both $(\tau,x)$ and
$(\tau+\delta,x)$ therefore belong to $\mathcal C$ when
$(\tau,x)\in Q:=(\tau_1,\tau_2-\delta)\times(x_0,\ln K_c)$.
Define

\[V_\delta(\tau,x)
:= V(\tau+\delta,x) - V(\tau,x)
\quad \text{for all }\
(\tau,x)\in \Omega_{T-\delta}.\]
Then, $V_\delta$ satisfies
$\partial_\tau V_\delta - \mathcal{L} V_\delta = 0$ in $Q$.
Note that there exists a point $(\tau',x')\in Q$ such that $V_\delta(\tau',x')=0.$ Otherwise, we see from the parabolic Hopf lemma that
\begin{equation}\label{Vx delta}
    \partial_x V_\delta(\tau,x_0)>0 
    \quad \text{for \ all}\ \ \tau\in(\tau_1,\tau_2-\delta).  
\end{equation}
But $V(\tau,x_0)=K_p-e^{x_0}$ and the continuity of $\partial_x V$, induced by the Sobolev embedding theorem, implies that
$\partial_x V_\delta(\tau,x_0)=0$ for all $\tau\in(\tau_1,\tau_2-\delta)$, which contradicts \eqref{Vx delta}.

Now, by the estimate in Lemma \ref{DV} and the continuity of $V$, it follows that $V_\delta(\tau,x)\geq 0$ for all $(\tau,x)\in Q$. Thus, $V_\delta$ attains its minimum as zero at $(\tau',x')\in Q$.
Furthermore, since $\partial_\tau V_\delta-\mathcal{L}V_\delta=0$ in $Q$, we deduce from the Schauder theory that $V_\delta$ is smooth in $Q$.
Applying the strong maximum principle \cite[Theorem 2.7]{LIE96} to $V_\delta$ yields
\begin{equation}\label{Vdelta}
V_\delta= 0\quad\text{in }\ \{(\tau,x)\in Q: \tau\le \tau'\}.    
\end{equation}

Both time arguments of $V_\delta$ lie in continuation on
$Q_\delta':=(0,\tau')\times(\ln K_p,\ln K_c)$, so
$(\partial_\tau-\mathcal L)V_\delta=0$ in $Q_\delta'$. Equation~\eqref{Vdelta}
and the strong maximum principle give $V_\delta=0$ on $Q_\delta'$.
In particular this holds on the fixed cylinder
$Q':=(0,\tau_1)\times(\ln K_p,\ln K_c)$ for every such $\delta$.
Dividing by $\delta$ and letting $\delta\downarrow0$ yields
$V_\tau=0$ on $Q'$, so $V$ is independent of $\tau$ in such region.
Considering the initial condition $V(0,x) = \max\{C(0,x), P(0,x)\}$ and the regularity $V \in C(\overline{\Omega_T})$, it follows that
\begin{equation}\label{Q'}
    V(\tau,x) = \max\{C(0,x), P(0,x)\} \quad \text{for all } (\tau,x) \in Q'.
\end{equation}
Since $C$ and $P$ are non-decreasing with respect to $\tau$ by Lemma \ref{lem:inequalityCP} (iii), we have
\[ 
\max\{C(0,x), P(0,x)\} \le \max\{C(\tau,x), P(\tau,x)\} \quad \text{for all } (\tau,x) \in Q'. 
\]
However, the inclusion $Q' \subset \mathcal{C}$ yields
\[ 
V(\tau,x) > \max\{C(\tau,x), P(\tau,x)\} \ge \max\{C(0,x), P(0,x)\} \quad \text{for all } (\tau,x) \in Q', 
\]
and this contradicts \eqref{Q'} and completes the proof.
\end{proof}

\begin{proof}[Proof of ii)]
We define $x_p^{\rm ch}(0)$ by
\[x_p^{\rm ch}(0) := \sup\{x \in \mathbb{R} : V(0,x) = K_p - e^x \}.\]
Since $V(0,x) = \max\{C(0,x), P(0,x)\}$ for all $x \in \mathbb{R}$, it is straightforward to verify that
\[x_p^{\rm ch}(0) = \min\{\bar{x}, x_p(0)\}.\]
We first observe that $\displaystyle\lim_{\tau \to 0^+} x_p^{\rm ch}(\tau)$ exists, since $x_p^{\rm ch}(\tau)$ is monotonically decreasing and bounded from above.

Suppose $\displaystyle{\lim_{\tau \to 0^+}x_p^{\rm ch}(\tau)}
> x_p^{\rm ch}(0)$.
Then there exists some $\tau_p>0$ such that
$x_p^{\rm ch}(\tau_p)>x_p^{\rm ch}(0)$.
Hence,
\[
K_p - e^{x_p^{\rm ch}(\tau_p)}
= V\big(\tau_p, x_p^{\rm ch}(\tau_p)\big)
\ge J\big(\tau_p, x_p^{\rm ch}(\tau_p)\big)
\ge J\big(0, x_p^{\rm ch}(\tau_p)\big)
> K_p - e^{x_p^{\rm ch}(\tau_p)},
\]
where $J=\max\{C,P\}$ and the second inequality follows from the result
$\partial_\tau C \geq 0$ and $\partial_\tau P \geq 0$ in $\Omega_T$
(see Lemma \ref{lem:inequalityCP} iii)).
Therefore, a contradiction occurs.

On the other hand, if $\displaystyle{\lim_{\tau \to 0^+}x_p^{\rm ch}(\tau)}$ 
is strictly less than $x_p^{\rm ch}(0)$,
there exist some $\varepsilon>0$ and $\delta>0$ such that
for all $\tau\in(0,\delta)$,
\[x^{\rm ch}_p(0)-x^{\rm ch}_p(\tau)\ge \varepsilon.
\]
Now, consider a cylinder $Q:=(0,\delta)\times(x^{\rm ch}_p(0)-\frac{\varepsilon}{2}, x^{\rm ch}_p(0))$.
By construction, we see that $Q$ is contained in the continuation region $\mathcal{C}$ and thus, $\partial_\tau V -\mathcal{L}V=0$ in $Q$.
Moreover, since $x^{\rm ch}_p(0) < \ln({\frac{r}{q}}K_p)$,
it follows that
\[(\partial_\tau-\mathcal{L})(V-(K_p-e^x))
\le qe^{x^{\rm ch}_p(0)}-rK_p
< 0 \quad \text{in}\ \ Q.
\]
Note that since $V=K_p-e^x$ on the line segment
$\mathcal{B}_Q:=\{0\}\times(x^{\rm ch}_p(0)-\frac{\varepsilon}{2}, x^{\rm ch}_p(0))$,
we apply the result in \cite[Theorem 19, p.~321]{FRI64} to $V-(K_p-e^x)$ to deduce that $V$ is smooth in $Q$ and up to the initial boundary $\mathcal{B}_Q$.
Thus, letting $\tau\to0^+$ in the above inequality yields
$\partial_\tau V<0$ on $\mathcal B_Q$.
This contradicts the first estimate in Lemma \ref{DV}.
Therefore, we conclude $\displaystyle{\lim_{\tau \to 0^+}x_p^{\rm ch}(\tau)}=x_p^{\rm ch}(0)$.
\end{proof}

By applying the same argument, analogous  results for  $x_c^{\rm ch}(\tau)$ can be obtained. Hence we omit the proof.

\begin{theorem}\label{thm:free-c} 
Let  $x_c^{\rm ch}(\tau)$ be defined as in  \eqref{FB2-2}. Then,
\begin{itemize}
\item[i)]$x_c^{\rm ch}$ is strictly increasing on $(0,T)$.

\item[ii)]$\displaystyle{\lim_{\tau \to 0^+}x_c^{\rm ch}(\tau)}
=\max\{\bar{x},{x_c}(0)\}$,
where $\bar{x}$ is the point stated in \eqref{barx}.
\end{itemize}
\end{theorem}

\begin{figure}[ht]
    \centering
    \begin{subfigure}{0.6\textwidth}  
        \includegraphics[width=\linewidth]{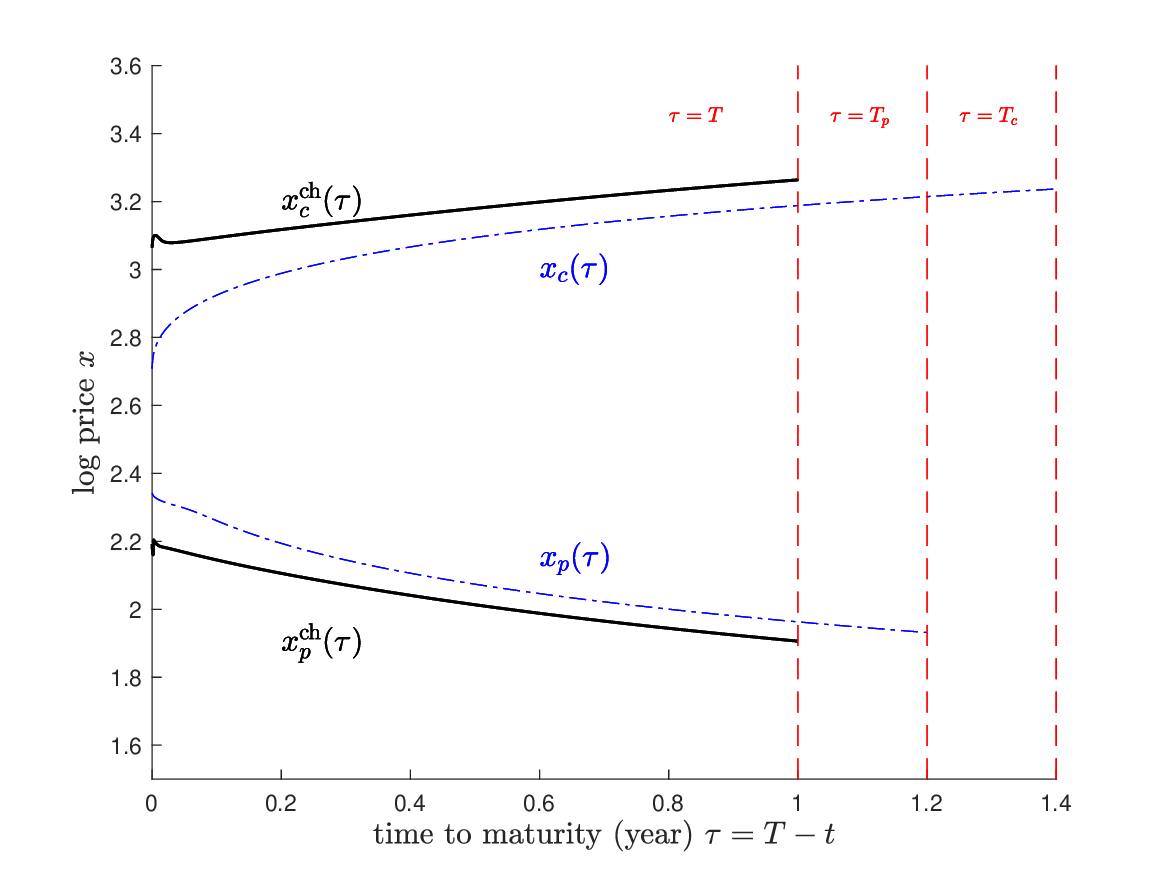}
     \end{subfigure}
    \caption{The free boundaries $x_c^{\rm ch}(\tau)$ and ${x}_p^{\rm ch}(\tau)$. }
    \label{fig:main2}
\end{figure}
Figure \ref{fig:main2} illustrates the behavior of the two free boundaries, $x_c^{\rm ch}(\tau)$ and $x_p^{\rm ch}(\tau)$, for $V(\tau, x)$. 
In this figure, we can observe that, as shown in Theorems \ref{thm:free-p} and \ref{thm:free-c}, $x_p^{\rm ch}(\tau)$ and $x_c^{\rm ch}(\tau)$ are strictly decreasing and strictly increasing in $\tau \in [0, T]$, respectively. Moreover, we can also verify the inequality \eqref{inq:frees}, which was discussed earlier.

{
\subsection{Regularity of the free boundary}

We conclude this section by establishing the local Lipschitz continuity
and higher regularity of the free boundaries
$x_p^{\rm ch}$ and $x_c^{\rm ch}$.
Note that
\begin{equation}\label{Cont reg}
    \mathcal{C}
    =
    \{(\tau,x)\in\Omega_T:
    x_p^{\rm ch}(\tau)<x<x_c^{\rm ch}(\tau)\}.
\end{equation}

{
{We first establish the continuity of $\partial_\tau V$ at every
point $(\tau_0,x_0)\in\partial\mathcal C\cap\Omega_T$ on the put side
$x_0<\ln K_p$. Continuity of $V$ and Lemma~\ref{FB1} imply
$V(\tau_0,x_0)=P(\tau_0,x_0)=K_p-e^{x_0}$, so}
\[
    {x_0}
    \leq x_p(\tau_0)
    \leq x_p(0)
    <
    \min\left\{
        \ln K_p,\,
        \ln\left(\frac{rK_p}{q}\right)
    \right\}.
\]
Choose a sufficiently small cylinder $Q$ centered at
{$(\tau_0,x_0)$} and a constant $\nu>0$ such that
\[
    x<\ln K_p,
    \qquad
    rK_p-qe^x\geq\nu
    \quad\text{in }Q.
\]
Set
\[
    u:=V-(K_p-e^x).
\]
By Lemma~\ref{FB1},
\[
    \{u>0\}\cap Q=\mathcal C\cap Q,
\]
and
\[
    \frac{\sigma^2}{2}u_{xx}
    +\left(r-q-\frac{\sigma^2}{2}\right)u_x
    -ru-u_\tau
    =
    (rK_p-qe^x)\mathbf1_{\{u>0\}}
    \quad\text{a.e. in }Q.
\]
Moreover, $u_\tau=\partial_\tau V\geq0$ a.e. by Lemma~\ref{DV}.
{The coefficients and the positive source are smooth. By
\cite[Theorem~1]{BDM2010erratum}, first shrink $Q$ to obtain
$u\in W_\infty^{1,2}(Q)$. Then
\cite[Theorem~1.2]{Blanchet2006-1} applies and gives continuity of
$\partial_\tau V$ and $\partial_\tau V=0$ on
$\partial\mathcal C\cap Q$. Also $u_x=0$ on the contact set,
by the continuity of $u_x$.}

{Similarly, at any $(\tau_0,x_0)\in\partial\mathcal C\cap\Omega_T$
with $x_0>\ln K_c$, continuity of $V$ and Lemma~\ref{FB1} give
$V(\tau_0,x_0)=C(\tau_0,x_0)=e^{x_0}-K_c$ and}
\[
    {x_0}
    \geq x_c(\tau_0)
    \geq x_c(0)
    >
    \max\left\{
        \ln K_c,\,
        \ln\left(\frac{rK_c}{q}\right)
    \right\}.
\]
Choose a sufficiently small cylinder $Q$ centered at
{$(\tau_0,x_0)$} and a constant $\nu>0$ such that
\[
    x>\ln K_c,
    \qquad
    qe^x-rK_c\geq\nu
    \quad\text{in }Q.
\]
Applying the same argument to
\[
    u:=V-(e^x-K_c)
\]
gives the continuity of $\partial_\tau V$ in $Q$ and
$\partial_\tau V=0$ on $\partial\mathcal C\cap Q$.
}

Thus the zero traces hold on the entire contact boundary on each side,
including possible horizontal portions; no continuity of the boundary
graphs has been used. Each $(\tau,x_i^{\rm ch}(\tau))$, $i=p,c$,
is a contact point approached spatially by continuation points at the
same time, and hence belongs to $\partial\mathcal C$. In particular,

\begin{equation}\label{Vtau boundary}
    \partial_\tau V(\tau,x_p^{\rm ch}(\tau))
    =
    \partial_\tau V(\tau,x_c^{\rm ch}(\tau))
    =0,
    \qquad \tau\in(0,T).
\end{equation}

We also recall from Lemma~\ref{DV} that
\[
    \partial_xV+e^x\geq0,
    \qquad
    e^x-\partial_xV\geq0.
\]
In the continuation region,
\[
    (\partial_\tau-\mathcal L)(\partial_xV+e^x)
    =qe^x,
    \qquad
    (\partial_\tau-\mathcal L)(e^x-\partial_xV)
    =qe^x.
\]
Since $q>0$, the strong maximum principle yields
\begin{equation}\label{strict spatial derivative}
    \partial_xV+e^x>0,
    \qquad
    e^x-\partial_xV>0
    \qquad\text{in }\mathcal C.
\end{equation}

\begin{lemma}\label{derivative estimate FB}
For every compact interval $I\Subset(0,T)$, there exists a constant
$C_I>0$ such that
\begin{equation}\label{put derivative estimate}
    0\leq\partial_\tau V(\tau,x)
    \leq
    C_I\bigl(\partial_xV(\tau,x)+e^x\bigr)
\end{equation}
for
\[
    \tau\in I,
    \qquad
    x_p^{\rm ch}(\tau)<x\leq\ln K_p,
\]
and
\begin{equation}\label{call derivative estimate}
    0\leq\partial_\tau V(\tau,x)
    \leq
    C_I\bigl(e^x-\partial_xV(\tau,x)\bigr)
\end{equation}
for
\[
    \tau\in I,
    \qquad
    \ln K_c\leq x<x_c^{\rm ch}(\tau).
\]
\end{lemma}

\begin{proof}
Let $I=[\tau_1,\tau_2]\Subset(0,T)$ and choose $\eta>0$ sufficiently
small so that
\[
    0<\tau_1-\eta<\tau_2+\eta<T.
\]

We first consider $x_p^{\rm ch}$. Define
\[
    D_\eta^p
    :=
    \{(\tau,x):
    \tau_1-\eta<\tau<\tau_2+\eta,\;
    x<\ln K_p,\;
    V(\tau,x)>K_p-e^x\}.
\]
From the definition of $x_p^{\rm ch}$ with Lemma~\ref{FB: contact}, the set $D_\eta^p$ is bounded.
Moreover, Lemma~\ref{FB1} gives
$\mathcal E_P^{\rm ch}=\{V=K_p-e^x\}$ and
$\mathcal E_C^{\rm ch}\subset\{x>\ln K_c\}$.
Furthermore, the continuity of $V$ and the strict inequality $V>K_p-e^x$ implies that $D_\eta^p$ is open.
Since $x<\ln K_p<\ln K_c$ and $V>K_p-e^x$ in $D_\eta^p$,
we have $D_\eta^p\subset\mathcal C$. Hence,
$\partial_\tau V-\mathcal LV=0$ in $D_\eta^p$, and the interior
Schauder estimates imply that $V$ is smooth in such region.

Since
$x_p^{\rm ch}(\tau)<\ln K_p<x_c^{\rm ch}(\tau)$,
the line $x=\ln K_p$ lies strictly inside the continuation region $\mathcal C$.
By \eqref{strict spatial derivative} and the continuity of
$\partial_xV$,
$\partial_xV(\tau,\ln K_p)+K_p$ is bounded below by a positive
constant on $[\tau_1-\eta,\tau_2+\eta]$.
Moreover, $\overline{D_\eta^p}$ is a compact subset of $\Omega_T$.
The continuity of $\partial_\tau V$ at the free boundary established
above, together with the interior Schauder estimates in $\mathcal C$,
implies that $\partial_\tau V$ is bounded in $D_\eta^p$.
Since $D_\eta^p$ is bounded and $q>0$, $qe^x$ is also bounded below
there by a positive constant.

Set
\[
    H_p
    :=
    (\tau-\tau_1+\eta)\partial_\tau V
    -
    M(\partial_xV+e^x).
\]
Since
\[
    (\partial_\tau-\mathcal L)H_p
    =
    \partial_\tau V-Mqe^x,
\]
we may choose $M>0$ sufficiently large so that
$(\partial_\tau-\mathcal L)H_p\leq0$ in $D_\eta^p$ and
$H_p\leq0$ on the right lateral boundary $x=\ln K_p$.

On $\tau=\tau_1-\eta$, we have
$H_p=-M(\partial_xV+e^x)\leq0$.
On the free boundary part of $D_\eta^p$,
\eqref{Vtau boundary} and the continuity of $\partial_xV$ give
$\partial_\tau V=0$ and $\partial_xV+e^x=0$, and hence $H_p=0$.
Therefore, the classical parabolic maximum principle
\cite[Corollary 2.5]{LIE96} yields $H_p\leq0$ in $D_\eta^p$.
Since $\tau-\tau_1+\eta\geq\eta$ for $\tau\in I$, we obtain
\[
    0\leq\partial_\tau V(\tau,x)
    \leq
    \frac{M}{\eta}
    \bigl(\partial_xV(\tau,x)+e^x\bigr),
\]
which proves \eqref{put derivative estimate}.

We next consider $x_c^{\rm ch}$. Define
\[
    D_\eta^c
    :=
    \{(\tau,x):
    \tau_1-\eta<\tau<\tau_2+\eta,\;
    x>\ln K_c,\;
    V(\tau,x)>e^x-K_c\}.
\]
From the definition of $x_c^{\rm ch}$ with Lemma~\ref{FB: contact}, the set $D_\eta^c$ is bounded.
Moreover, Lemma~\ref{FB1} gives
$\mathcal E_C^{\rm ch}=\{V=e^x-K_c\}$ and
$\mathcal E_P^{\rm ch}\subset\{x<\ln K_p\}$.
Furthermore, the continuity of $V$ and the strict inequality $V>e^x-K_C$ implies that $D_\eta^c$ is open.
Since $x>\ln K_c>\ln K_p$ and $V>e^x-K_c$ in $D_\eta^c$,
we have $D_\eta^c\subset\mathcal C$. Hence,
$\partial_\tau V-\mathcal LV=0$ in $D_\eta^c$, and the interior
Schauder estimates imply that $V$ is smooth in such region.

Since
$x_p^{\rm ch}(\tau)<\ln K_c<x_c^{\rm ch}(\tau)$,
the line $x=\ln K_c$ lies strictly inside $\mathcal C$.
By \eqref{strict spatial derivative} and the continuity of
$\partial_xV$,
$K_c-\partial_xV(\tau,\ln K_c)$ is bounded below by a positive
constant on $[\tau_1-\eta,\tau_2+\eta]$.
Moreover, the continuity of $\partial_\tau V$ at the free boundary
established above, together with the interior Schauder estimates in
$\mathcal C$, implies that $\partial_\tau V$ is bounded in
$D_\eta^c$. Since $D_\eta^c$ is bounded and $q>0$, $qe^x$ is also
bounded below there by a positive constant.

Set
\[
    H_c
    :=
    (\tau-\tau_1+\eta)\partial_\tau V
    -
    M(e^x-\partial_xV).
\]
Since
\[
    (\partial_\tau-\mathcal L)H_c
    =
    \partial_\tau V-Mqe^x,
\]
we may choose $M>0$ sufficiently large so that
$(\partial_\tau-\mathcal L)H_c\leq0$ in $D_\eta^c$ and
$H_c\leq0$ on the left lateral boundary $x=\ln K_c$.

On $\tau=\tau_1-\eta$, we have
$H_c=-M(e^x-\partial_xV)\leq0$.
On the free boundary part of $D_\eta^c$,
\eqref{Vtau boundary} and the continuity of $\partial_xV$ give
$\partial_\tau V=0$ and $e^x-\partial_xV=0$, and hence $H_c=0$.
Therefore, the classical parabolic maximum principle
\cite[Corollary 2.5]{LIE96} yields $H_c\leq0$ in $D_\eta^c$.
Since $\tau-\tau_1+\eta\geq\eta$ for $\tau\in I$, we obtain
\[
    0\leq\partial_\tau V(\tau,x)
    \leq
    \frac{M}{\eta}
    \bigl(e^x-\partial_xV(\tau,x)\bigr),
\]
which proves \eqref{call derivative estimate}.
Enlarging $M$ if necessary and taking $C_I=M/\eta$ completes the proof.
\end{proof}

We now apply the preceding estimate to level curves approaching the
two free boundaries.}

{
\begin{theorem}\label{Lipschitz free boundary}
Let $x_p^{\rm ch}$ and $x_c^{\rm ch}$ be defined as in
\eqref{FB2-1} and \eqref{FB2-2}, respectively.
Then, both $x_p^{\rm ch}$ and $x_c^{\rm ch}$ are locally Lipschitz
continuous on $(0,T)$.
In particular,
\[
    \partial\mathcal C\cap\Omega_T
    =
    \{(\tau,x_p^{\rm ch}(\tau)):\tau\in(0,T)\}
    \cup
    \{(\tau,x_c^{\rm ch}(\tau)):\tau\in(0,T)\}.
\]
\end{theorem}

\begin{proof}
Fix a compact interval $I\Subset(0,T)$.
Since $x=\ln K_p$ lies strictly in the continuation region,
$\min_{\tau\in I}V(\tau,\ln K_p)>0$.
Hence, for all sufficiently large $n$, there exists a unique
$x_{p,n}(\tau)$ satisfying
\begin{equation}\label{put level curve}
    V(\tau,x_{p,n}(\tau))
    -(K_p-e^{x_{p,n}(\tau)})
    =\frac1n,
\end{equation}
with
$x_p^{\rm ch}(\tau)<x_{p,n}(\tau)<\ln K_p$.
The uniqueness follows from the strict monotonicity
$\partial_xV+e^x>0$ in $\mathcal C$.
By the implicit function theorem,
$x_{p,n}\in C^1(I)$, and differentiating
\eqref{put level curve} gives
\[
    x_{p,n}'(\tau)
    =
    -
    \frac{\partial_\tau V(\tau,x_{p,n}(\tau))}
    {\partial_xV(\tau,x_{p,n}(\tau))
     +e^{x_{p,n}(\tau)}},
\]
since the denominator is strictly positive.
Lemma~\ref{derivative estimate FB} therefore yields
$|x_{p,n}'(\tau)|\leq C_I$ for all $\tau\in I$, where
$C_I>0$ is independent of $n$.
Hence, for any $\tau,s\in I$,
\[
    |x_{p,n}(\tau)-x_{p,n}(s)|
    \leq C_I|\tau-s|,
\]
where $C_I$ is independent of $n$. {
For fixed $\tau$, since $x_{p,n}(\tau)$ decreases in $n$, let
\[
    \widehat x_p:=\lim_{n\to\infty}x_{p,n}(\tau)
    \geq x_p^{\rm ch}(\tau).
\]
By the continuity of $V$, letting $n\to \infty$ at \eqref{put level curve} yields
\[
    V(\tau,\widehat x_p)=K_p-e^{\widehat x_p}.
\]
Hence, by the definition of $x_p^{\rm ch}(\tau)$ in
\eqref{FB2-1}, we have $\widehat x_p\leq x_p^{\rm ch}(\tau)$.
Therefore, $\widehat x_p=x_p^{\rm ch}(\tau)$ and hence
\[
    x_{p,n}(\tau)\downarrow x_p^{\rm ch}(\tau)
    \quad \text{as }\ n\to \infty.
\]}
Thus, passing $n\to\infty$ in the above inequality, we obtain
\[
    |x_p^{\rm ch}(\tau)-x_p^{\rm ch}(s)|
    =
    \lim_{n\to\infty}
    |x_{p,n}(\tau)-x_{p,n}(s)|
    \leq C_I|\tau-s|.
\]
Therefore, $x_p^{\rm ch}$ is Lipschitz continuous on $I$.

Similarly, since $x=\ln K_c$ lies strictly in the continuation
region, for all sufficiently large $n$ there exists a unique
$x_{c,n}(\tau)$ satisfying
\begin{equation}\label{call level curve}
    V(\tau,x_{c,n}(\tau))
    -(e^{x_{c,n}(\tau)}-K_c)
    =\frac1n,
\end{equation}
with
$\ln K_c<x_{c,n}(\tau)<x_c^{\rm ch}(\tau)$.
Differentiating \eqref{call level curve} gives
\[
    x_{c,n}'(\tau)
    =
    \frac{\partial_\tau V(\tau,x_{c,n}(\tau))}
    {e^{x_{c,n}(\tau)}
     -\partial_xV(\tau,x_{c,n}(\tau))}.
\]
since the denominator is strictly positive.
Lemma~\ref{derivative estimate FB} therefore yields
$|x_{c,n}'(\tau)|\leq C_I$ for all $\tau\in I$, where
$C_I>0$ is independent of $n$.
Hence, for any $\tau,s\in I$,
\[
    |x_{c,n}(\tau)-x_{c,n}(s)|
    \leq C_I|\tau-s|,
\]
{
For fixed $\tau$, since $x_{c,n}(\tau)$ increases in $n$, let
\[
    \widehat x_c:=\lim_{n\to\infty}x_{c,n}(\tau)
    \leq x_c^{\rm ch}(\tau).
\]
By the continuity of $V$, letting $n\to\infty$ at \eqref{call level curve} yields
\[
    V(\tau,\widehat x_c)=e^{\widehat x_c}-K_c.
\]
Hence, by the definition of $x_c^{\rm ch}(\tau)$ in
\eqref{FB2-2}, we have $\widehat x_c\geq x_c^{\rm ch}(\tau)$.
Therefore, $\widehat x_c= x_c^{\rm ch}(\tau)$ and hence
\[
    x_{c,n}(\tau)\uparrow x_c^{\rm ch}(\tau)
    \quad \text{as }\ n\to \infty.
\]
} Applying Lemma~\ref{derivative estimate FB} and then taking
$n\to\infty$ yields
\[
    |x_c^{\rm ch}(\tau)-x_c^{\rm ch}(s)|
    \leq C_I|\tau-s|,
    \qquad \tau,s\in I.
\]
Therefore, both free boundaries are locally Lipschitz continuous.

Since $x_p^{\rm ch}$ and $x_c^{\rm ch}$ are continuous, if
$x_0<x_p^{\rm ch}(\tau_0)$ or
$x_0>x_c^{\rm ch}(\tau_0)$, then there exists a sufficiently small
parabolic cylinder centered at $(\tau_0,x_0)$ which is contained in
the corresponding contact set. Similarly, if
\[
    x_p^{\rm ch}(\tau_0)<x_0<x_c^{\rm ch}(\tau_0),
\]
then a sufficiently small parabolic cylinder centered at
$(\tau_0,x_0)$ is contained in $\mathcal C$.

Hence, every point in $\Omega_T$ which does not lie on either graph
belongs to the interior of either $\mathcal C$ or its complement.
{
On the other hand, by the definitions of $x_p^{\rm ch}$ and
$x_c^{\rm ch}$, the continuity of $V$, and \eqref{Cont reg},
every point on
$x=x_p^{\rm ch}(\tau)$ or $x=x_c^{\rm ch}(\tau)$
belongs to $\partial\mathcal C$.}
Therefore,
\[
    \partial\mathcal C\cap\Omega_T
    =
    \{(\tau,x_p^{\rm ch}(\tau)):\tau\in(0,T)\}
    \cup
    \{(\tau,x_c^{\rm ch}(\tau)):\tau\in(0,T)\}.
\]
\end{proof}

To improve the regularity of the free boundaries, we first establish
the nondegeneracy of the corresponding spatial derivatives.

\begin{lemma}\label{FB nondeg}
Let $I\Subset(0,T)$ be a compact interval.
Then there exist constants $\delta>0$ and $c>0$ such that
\begin{align}
    \partial_xV(\tau,x)+e^x
    &\geq
    c\bigl(x-x_p^{\rm ch}(\tau)\bigr),
    &&x_p^{\rm ch}(\tau)<x<x_p^{\rm ch}(\tau)+\delta,
    \label{put nondeg}\\
    e^x-\partial_xV(\tau,x)
    &\geq
    c\bigl(x_c^{\rm ch}(\tau)-x\bigr),
    &&x_c^{\rm ch}(\tau)-\delta<x<x_c^{\rm ch}(\tau),
    \label{call nondeg}
\end{align}
for all $\tau\in I$.
\end{lemma}

\begin{proof}
We first consider $x_p^{\rm ch}$. By \eqref{inq:frees} and the
properties of the American put boundary $x_p$,
\[
    x_p^{\rm ch}(\tau)
    \leq x_p(\tau)
    \leq x_p(0)
    <
    \ln\left(\frac{rK_p}{q}\right).
\]
Hence, there exists $\nu_p>0$ such that
$rK_p-qe^{x_p^{\rm ch}(\tau)}\geq\nu_p$ for all $\tau\in I$.

Set $W:=V-(K_p-e^x)$. Then $W\geq0$ and, in the continuation
region near $x_p^{\rm ch}$,
\[
    \partial_\tau W-\mathcal LW=qe^x-rK_p,
\]
while $W=\partial_xW=0$ on $x=x_p^{\rm ch}(\tau)$.

Since $x_p^{\rm ch}$ is locally Lipschitz continuous, its graph over
$I$ is contained in a compact subset $\mathcal N\Subset\Omega_T$.
{
On $\mathcal N$ the derivative $\partial_xV$ is H\"older continuous:
by Theorem~\ref{thm:strong-solution} we may fix $p>3$, and the
parabolic Sobolev embedding \cite[Chapter~II, Lemma~3.3]{LSU68} applied
to $V\in W^{1,2}_p$ on a fixed neighbourhood of $\mathcal N$ gives
$\partial_xV\in C^{\frac{\alpha}{2},\alpha}(\mathcal N)$ with
$\alpha=1-3/p\in(0,1)$ and a norm bound depending only on
$\mathcal N$; in particular the H\"older constant is the same for all
$\tau\in I$.}
Therefore, this H\"older continuity, together with
$\partial_xW=0$ on the free boundary $x_p^{\rm ch}(\tau)$, implies that
$\partial_xW(\tau,x)\to0$ as
$x\to x_p^{\rm ch}(\tau)^+$, uniformly for $\tau\in I$.
More precisely, for some $C>0$ and $\alpha\in(0,1)$,
\[
    0\leq\partial_xW(\tau,x)
    \leq C\bigl(x-x_p^{\rm ch}(\tau)\bigr)^\alpha
\]
whenever $x>x_p^{\rm ch}(\tau)$ is sufficiently close to
$x_p^{\rm ch}(\tau)$, where $C$ is independent of $\tau\in I$.

By Lemma~\ref{derivative estimate FB},
$0\leq\partial_\tau W\leq C_I\partial_xW$ in the same region sufficiently close to $x_p^{\rm ch}(\tau)$.
Consequently, $\partial_\tau W(\tau,x)\to0$ uniformly for
$\tau\in I$ as $x\to x_p^{\rm ch}(\tau)^+$. Moreover, since
$W=0$ on the boundary, integrating the preceding estimate for
$\partial_xW$ shows that $W(\tau,x)\to0$ uniformly as well.

Rewriting the equation for $W$ gives
\[
    \frac{\sigma^2}{2}\partial_{xx}W
    =
    \partial_\tau W
    -
    \left(r-q-\frac{\sigma^2}{2}\right)\partial_xW
    +rW+rK_p-qe^x.
\]
Thus, as $x\to x_p^{\rm ch}(\tau)^+$, the right-hand side converges
uniformly for $\tau\in I$ to
$rK_p-qe^{x_p^{\rm ch}(\tau)}$. Since this quantity is bounded below
by $\nu_p$ uniformly on $I$, there exist $\delta_p>0$ and $c_p>0$,
independent of $\tau\in I$, such that
\[
    \partial_{xx}W(\tau,x)\geq c_p
\]
whenever
$x_p^{\rm ch}(\tau)<x<x_p^{\rm ch}(\tau)+\delta_p$.
{
Since $W$ is smooth in the continuation region by the interior
Schauder estimates, for each fixed $\tau\in I$ and
$x_p^{\rm ch}(\tau)<y<x<x_p^{\rm ch}(\tau)+\delta_p$,
\[
    \partial_xW(\tau,x)-\partial_xW(\tau,y)
    =
    \int_y^x\partial_{xx}W(\tau,z)\,dz
    \geq c_p(x-y).
\]
Letting $y\to x_p^{\rm ch}(\tau)^+$ and using the continuity of
$\partial_xW$ together with
$\partial_xW(\tau,x_p^{\rm ch}(\tau))=0$ yields
\eqref{put nondeg}.}

We next consider $x_c^{\rm ch}$. By \eqref{inq:frees} and the
properties of the American call boundary $x_c$,
\[
    x_c^{\rm ch}(\tau)
    \geq x_c(\tau)
    \geq x_c(0)
    >
    \ln\left(\frac{rK_c}{q}\right).
\]
Hence, there exists $\nu_c>0$ such that
$qe^{x_c^{\rm ch}(\tau)}-rK_c\geq\nu_c$ for all $\tau\in I$.

Set $W:=V-(e^x-K_c)$. Then $W\geq0$ and
$\partial_\tau W-\mathcal LW=rK_c-qe^x$ in the continuation region
near $x_c^{\rm ch}$, while
$W=\partial_xW=0$ on $x=x_c^{\rm ch}(\tau)$.
As above, the local H\"older continuity of $\partial_xV$ gives
\[
    0\leq-\partial_xW(\tau,x)
    \leq C\bigl(x_c^{\rm ch}(\tau)-x\bigr)^\alpha
\]
near $x_c^{\rm ch}$, uniformly for $\tau\in I$.
Lemma~\ref{derivative estimate FB} then implies
$0\leq\partial_\tau W\leq C_I(-\partial_xW)$, and hence both
$\partial_\tau W$ and $W$ converge uniformly to zero as
$x\to x_c^{\rm ch}(\tau)^-$.

Since
\[
    \frac{\sigma^2}{2}\partial_{xx}W
    =
    \partial_\tau W
    -
    \left(r-q-\frac{\sigma^2}{2}\right)\partial_xW
    +rW+qe^x-rK_c,
\]
the right-hand side converges uniformly to
$qe^{x_c^{\rm ch}(\tau)}-rK_c$, which is bounded below by $\nu_c$
uniformly on $I$. Therefore, there exist $\delta_c>0$ and $c_c>0$,
independent of $\tau\in I$, such that
\[
    \partial_{xx}W(\tau,x)\geq c_c
\]
whenever
$x_c^{\rm ch}(\tau)-\delta_c<x<x_c^{\rm ch}(\tau)$.
Since $W$ is smooth in the continuation region, integrating with
respect to $x$ and using
$\partial_xW(\tau,x_c^{\rm ch}(\tau))=0$, exactly as above, yields
\eqref{call nondeg}.

Finally, taking
$\delta=\min\{\delta_p,\delta_c\}$ and
$c=\min\{c_p,c_c\}$ completes the proof.
\end{proof}

We next apply the boundary Harnack inequality to improve the
regularity of the free boundaries.

{
We state the two local results in the one-dimensional form needed here.
The notation $C_p^\beta$ denotes parabolic H\"older regularity, with
one time derivative having weight two. Let
$\mathcal P_0=\partial_\tau-a\partial_{xx}-b\partial_x$ and
let $d$ be the distance to the local boundary in the metric
$|x-y|+|\tau-s|^{1/2}$.

\begin{bhfirst}[{\cite[Theorem~1.2]{Kukuljan2022}}]
Suppose $D$ is a parabolic $C_p^1$ graph domain, $a$ is continuous
and uniformly positive and bounded, and $b$ is bounded.
Let bounded $u_i$ satisfy $\mathcal P_0u_i=f_i$ in $D$ and
$u_i=0$ on its local boundary, with bounded $f_i$.
If $|u_2|\geq c_2d$ and
$\|u_2\|_\infty+\|f_2\|_\infty\leq C_2$, then
$u_1/u_2\in C_p^{1-\varepsilon}$ up to the boundary of every smaller
cylinder, for each $0<\varepsilon<1$.
\end{bhfirst}

\begin{bhhigher}[{\cite[Theorem~1.3]{Kukuljan2022}}]
Let $\beta>1$ be noninteger. Under the same ellipticity, boundary,
and nondegeneracy conditions, assume that $D$ is $C_p^\beta$ and
$a,b,f_i\in C_p^{\beta-1}$. With the denominator's data bounded
in these norms, $u_1/u_2\in C_p^\beta$ on every smaller cylinder.
\end{bhhigher}
Both conclusions include local norm estimates in terms of the stated
data, so their constants are independent of any level-curve index.
}

\begin{theorem}\label{C1alpha free boundary}
For every $\alpha\in(0,\frac12)$,
\[
    x_p^{\rm ch},\,x_c^{\rm ch}
    \in C_{\rm loc}^{1,\alpha}((0,T)).
\]
\end{theorem}

\begin{proof}
We first consider $x_p^{\rm ch}$.
Fix $\tau_0\in(0,T)$ and set
\[
    x_0:=x_p^{\rm ch}(\tau_0).
\]
Choose a compact interval $I\Subset(0,T)$ containing $\tau_0$ in its
interior. By Lemma~\ref{FB nondeg}, there exist $\delta>0$ and $c>0$
such that
\[
    \partial_xV(\tau,x)+e^x
    \geq
    c\bigl(x-x_p^{\rm ch}(\tau)\bigr)
\]
for $\tau\in I$ and
\[
    x_p^{\rm ch}(\tau)
    <x<
    x_p^{\rm ch}(\tau)+\delta.
\]

Let $L$ be a Lipschitz constant of $x_p^{\rm ch}$ on $I$, and define
the parabolic cylinder
\[
    Q_r(\tau_0,x_0)
    :=
    (\tau_0-r^2,\tau_0+r^2)
    \times(x_0-r,x_0+r).
\]
Choose $r>0$ sufficiently small so that
\[
    [\tau_0-4r^2,\tau_0+4r^2]\subset I,
    \qquad
    2r+4Lr^2<\delta,
\]
and
\[
    Q_{2r}(\tau_0,x_0)\cap\mathcal E_C^{\rm ch}
    =\emptyset.
\]
Then
\[
    Q_{2r}(\tau_0,x_0)\cap\mathcal C
    =
    Q_{2r}(\tau_0,x_0)
    \cap
    \{x>x_p^{\rm ch}(\tau)\}.
\]
Moreover, for every $(\tau,x)$ in this set,
\[
\begin{aligned}
    0
    &<
    x-x_p^{\rm ch}(\tau) \\
    &\leq
    |x-x_0|
    +
    |x_p^{\rm ch}(\tau)-x_p^{\rm ch}(\tau_0)| \\
    &<
    2r+4Lr^2
    <\delta.
\end{aligned}
\]
Consequently,
\[
    \partial_xV(\tau,x)+e^x
    \geq
    c\bigl(x-x_p^{\rm ch}(\tau)\bigr)
    \quad\text{in }
    Q_{2r}(\tau_0,x_0)\cap\mathcal C.
\]

In $Q_{2r}(\tau_0,x_0)\cap\mathcal C$, we have
\[
    (\partial_\tau-\mathcal L)(\partial_\tau V)=0,
    \qquad
    (\partial_\tau-\mathcal L)(\partial_xV+e^x)=qe^x.
\]
By the continuity of $\partial_\tau V$ and $\partial_xV$ up to
$x=x_p^{\rm ch}(\tau)$ from the continuation region, we have
\[
    \partial_\tau V=0,
    \qquad
    \partial_xV+e^x=0
    \quad\text{on }
    Q_{2r}(\tau_0,x_0)
    \cap\{x=x_p^{\rm ch}(\tau)\}.
\]

{
We check the hypotheses of Theorem~A one by one on a smaller cylinder.
\begin{enumerate}[label=(\roman*)]
\item \emph{Domain.} For $g(\tau)=x_p^{\rm ch}(\tau)$,
\[
 \frac{|g(\tau)-g(s)|}{|\tau-s|^{1/2}}
 \leq L|\tau-s|^{1/2}\longrightarrow0.
\]
There are no tangential spatial variables in one dimension. Thus a
time-Lipschitz graph is $C_p^1$ in the sense of
\cite[Definition~1.1]{Kukuljan2022}; in fact it belongs to
$C_p^\beta$ for every $1<\beta<2$.
\item \emph{Operator.} Put $a=\sigma^2/2$,
$b=r-q-\sigma^2/2$ and use
$u_1=e^{r\tau}V_\tau$ and
$u_2=e^{r\tau}(V_x+e^x)$. Then
$\mathcal P_0=\partial_\tau-a\partial_{xx}-b\partial_x$
has constant uniformly parabolic coefficients and no zero-order term.
\item \emph{Equations and bounds.} The preceding equations become
\[
 \mathcal P_0u_1=0,\qquad
 \mathcal P_0u_2=q e^{r\tau+x}.
\]
The sources are smooth and bounded locally. The local continuity of
$V_\tau$ and $V_x$ gives the required bounds for $u_1,u_2$.
\item \emph{Traces.} Equation~\eqref{Vtau boundary} and spatial
smooth fit give $u_1=u_2=0$ on the local boundary graph.
\item \emph{Nondegeneracy.} The same-time boundary point gives
$d(\tau,x)\leq x-g(\tau)$. Lemma~\ref{FB nondeg} therefore yields
$u_2\geq c_2d$, since $e^{r\tau}$ has a positive lower bound locally.
\end{enumerate}
Theorem~A now applies, and its quotient $u_1/u_2$ is exactly
$V_\tau/(V_x+e^x)$. Taking $\varepsilon=1-2\alpha$ gives the
following estimate.
}
For every $\alpha\in(0,\frac12)$, setting
\[
    F(\tau,x)
    :=
    \frac{\partial_\tau V(\tau,x)}
    {\partial_xV(\tau,x)+e^x},
\]
we obtain
\[
    |F(\tau,x)-F(s,y)|
    \leq
    C\bigl(|\tau-s|^\alpha+|x-y|^{2\alpha}\bigr)
\]
for $(\tau,x),(s,y)\in Q_r(\tau_0,x_0)\cap\mathcal C$.

Recall from \eqref{put level curve} that
\[
    x_{p,n}'(\tau)
    =
    -F(\tau,x_{p,n}(\tau)).
\]
Since $x_p^{\rm ch}$ is Lipschitz continuous on $I$ with constant $L$,
by taking $r>0$ sufficiently small if necessary, we have
\[
    x_p^{\rm ch}(\tau)
    \leq x_0+Lr^2
    <x_0+r
    \qquad\text{for }|\tau-\tau_0|\leq r^2.
\]
Hence, the right lateral boundary of
$Q_r(\tau_0,x_0)$ lies strictly in the continuation region. Therefore,
\[
    m_r:=
    \min_{|\tau-\tau_0|\leq r^2}
    \left[
        V(\tau,x_0+r)-\bigl(K_p-e^{x_0+r}\bigr)
    \right]>0.
\]
Choosing $n$ sufficiently large so that $1/n<m_r$, the level curve
$x_{p,n}$ lies in $Q_r(\tau_0,x_0)\cap\mathcal C$ for
$|\tau-\tau_0|\leq r^2$.

Moreover, combining Lemma~\ref{derivative estimate FB} with the
level-curve identity above gives a uniform Lipschitz bound for
$x_{p,n}$. Therefore,
\[
    |x_{p,n}(\tau)-x_{p,n}(s)|
    \leq C|\tau-s|,
\]
where $C>0$ is independent of $n$. Using the H\"older estimate for
$F$, we then obtain
\[
\begin{aligned}
    |x_{p,n}'(\tau)-x_{p,n}'(s)|
    &=
    \left|
        F(\tau,x_{p,n}(\tau))
        -F(s,x_{p,n}(s))
    \right| \\
    &\leq
    C\left(
        |\tau-s|^\alpha
        +|x_{p,n}(\tau)-x_{p,n}(s)|^{2\alpha}
    \right) \\
    &\leq
    C\left(
        |\tau-s|^\alpha+|\tau-s|^{2\alpha}
    \right) \\
    &\leq
    C|\tau-s|^\alpha.
\end{aligned}
\]
Since the uniform Lipschitz bound also gives
$\|x_{p,n}'\|_{L^\infty}\leq C$, it follows that
\[
    \|x_{p,n}'\|_{C^\alpha}\leq C.
\]
As $x_{p,n}$ lies in $Q_r(\tau_0,x_0)$, we consequently have
\[
    \|x_{p,n}\|_{C^{1,\alpha}}\leq C
\]
on $[\tau_0-r^2,\tau_0+r^2]$, with $C$ independent of $n$.

By the Arzel\`a--Ascoli theorem, a subsequence converges in $C^1$.
Since $x_{p,n}(\tau)\to x_p^{\rm ch}(\tau)$ for each $\tau$, the limit
is $x_p^{\rm ch}$. Passing to the limit in the preceding estimate
therefore gives
\[
    x_p^{\rm ch}\in C^{1,\alpha}
    \bigl((\tau_0-r^2,\tau_0+r^2)\bigr).
\]
Since $\tau_0\in(0,T)$ was arbitrary,
\[
    x_p^{\rm ch}\in C_{\rm loc}^{1,\alpha}((0,T)).
\]

We next consider $x_c^{\rm ch}$.
Fix $\tau_0\in(0,T)$ and set
\[
    x_0:=x_c^{\rm ch}(\tau_0).
\]
By Lemma~\ref{FB nondeg} and the local Lipschitz continuity of
$x_c^{\rm ch}$, we may choose $r>0$ sufficiently small so that
\[
    Q_{2r}(\tau_0,x_0)\cap\mathcal C
    =
    Q_{2r}(\tau_0,x_0)\cap\{x<x_c^{\rm ch}(\tau)\},
\]
and
\[
    e^x-\partial_xV(\tau,x)
    \geq
    c\bigl(x_c^{\rm ch}(\tau)-x\bigr)
    \quad\text{in }
    Q_{2r}(\tau_0,x_0)\cap\mathcal C.
\]
In this set,
\[
    (\partial_\tau-\mathcal L)(\partial_\tau V)=0,
    \qquad
    (\partial_\tau-\mathcal L)(e^x-\partial_xV)=qe^x,
\]
with
\[
    \partial_\tau V=0,
    \qquad
    e^x-\partial_xV=0
    \quad\text{on }
    Q_{2r}(\tau_0,x_0)\cap\{x=x_c^{\rm ch}(\tau)\}.
\]
{
For this side use $u_1=e^{r\tau}V_\tau$ and
$u_2=e^{r\tau}(e^x-V_x)$. The same operator, source, and trace
checks apply, while $d\leq x_c^{\rm ch}(\tau)-x$ gives the
required lower bound. Reflection of the spatial coordinate puts the
domain in the graph orientation of Theorem~A.
} Applying Theorem~A and repeating the preceding argument with the
level curves $x_{c,n}$, we obtain
\[
    x_c^{\rm ch}\in C_{\rm loc}^{1,\alpha}((0,T)).
\]
Since $\alpha\in(0,\frac12)$ was arbitrary, the proof is complete.
\end{proof}

Finally, we prove the higher regularity of the free boundaries.

\begin{corollary}\label{smooth free boundary}
The free boundaries satisfy
\[
    x_p^{\rm ch},\,x_c^{\rm ch}\in C^\infty((0,T)).
\]
\end{corollary}

\begin{proof}
{
Fix $\alpha\in(0,1/2)$ and use the fixed smaller cylinders of
Theorem~\ref{C1alpha free boundary}. Set
\[
 F_p=\frac{V_\tau}{V_x+e^x},\qquad
 F_c=\frac{V_\tau}{e^x-V_x}.
\]
By Theorem~\ref{C1alpha free boundary}, the boundary graphs are
$C^{1,\alpha}$ in time, hence $C_p^{2+2\alpha}$.
All coefficients and transformed sources are smooth, the zero
boundary traces and the lower bounds from Lemma~\ref{FB nondeg}
remain valid, and the denominator norms are bounded on the fixed
cylinders. Theorem~B with $\beta=2+2\alpha\in(2,3)$ therefore gives
\[
 F_p,F_c\in C_p^{2+2\alpha},\qquad
 \partial_\tau F_i\in C_p^{2\alpha},\quad
 \partial_xF_i\in C_p^{1+2\alpha},\quad i=p,c,
\]
up to their respective boundary graphs.

We give the first bootstrap step explicitly. The level curves are
continuously differentiable and lie in the continuation region and satisfy
$x_{p,n}'=-F_p(\tau,x_{p,n})$ and
$x_{c,n}'=F_c(\tau,x_{c,n})$. Differentiating gives
\[
\begin{aligned}
 x_{p,n}''(\tau)
 &=-\partial_\tau F_p(\tau,x_{p,n}(\tau))
   -\partial_xF_p(\tau,x_{p,n}(\tau))x_{p,n}'(\tau),\\
 x_{c,n}''(\tau)
 &=\partial_\tau F_c(\tau,x_{c,n}(\tau))
   +\partial_xF_c(\tau,x_{c,n}(\tau))x_{c,n}'(\tau).
\end{aligned}
\]
The curves have uniformly bounded $C^{1,\alpha}$ norms by the
preceding theorem. If $H\in C_p^{2\alpha}$, composition with these
uniformly Lipschitz curves satisfies
\[
 |H(\tau,x_{i,n}(\tau))-H(s,x_{i,n}(s))|
 \leq C\bigl(|\tau-s|^\alpha+
             |x_{i,n}(\tau)-x_{i,n}(s)|^{2\alpha}\bigr)
 \leq C'|\tau-s|^\alpha.
\]
This applies to $\partial_\tau F_i$ and, by the local embedding
$C_p^{1+2\alpha}\subset C_p^{2\alpha}$, to $\partial_xF_i$.
The displayed chain rules and the product estimate in $C^\alpha$
give $\|x_{i,n}''\|_{C^\alpha}\leq C$ uniformly in $n$.
Thus $\|x_{i,n}\|_{C^{2,\alpha}}\leq C$ locally. Compactness,
together with the already identified pointwise limits, yields
$x_i^{\rm ch}\in C^{2,\alpha}_{\rm loc}$.

For the induction, suppose that the boundary and the approximating
level curves have the corresponding local $C^{k,\alpha}$ bounds.
Their boundary graph is $C_p^{2k+2\alpha}$, so Theorem~B with
$\beta=2k+2\alpha\notin\mathbb N$ gives
$F_i\in C_p^{2k+2\alpha}$. Differentiating
$x_{i,n}'=\pm F_i(\tau,x_{i,n})$ $k$ times involves only
derivatives $\partial_\tau^j\partial_x^\ell F_i$ of parabolic
order $2j+\ell\leq2k$ and derivatives of $x_{i,n}$ of order
at most $k$. The former retain at least $C_p^{2\alpha}$
regularity, and the latter have the induction bounds. The same
composition and product estimates therefore give
$\|x_{i,n}^{(k+1)}\|_{C^\alpha}\leq C$.
Passing to the limit proves $C^{k+1,\alpha}_{\rm loc}$ regularity.
Induction establishes $x_p^{\rm ch},x_c^{\rm ch}\in C^\infty((0,T))$.
}
\label{rev:smooth-end}
\end{proof}

{
\begin{remark}
A key point in the regularity argument is that, by
Lemma~\ref{FB1}, the contact near $x_p^{\rm ch}$ and
$x_c^{\rm ch}$ occurs with the payoff parts $K_p-e^x$ and
$e^x-K_c$, respectively. Hence, subtracting the corresponding
obstacle gives the smooth right-hand sides $qe^x-rK_p$ and
$rK_c-qe^x$, respectively.

Moreover, the free boundaries lie strictly in the regions where
$rK_p-qe^x>0$ and $qe^x-rK_c>0$, respectively. These strict
inequalities give the nondegeneracy in Lemma~\ref{FB nondeg}, while
the smoothness of the right-hand sides allows us to apply the
higher-order boundary Harnack inequality.

If the contact occurred with the continuation part of the American
put or call, the corresponding obstacle would satisfy the same
homogeneous equation as $V$ in that region. Hence, the right-hand
side obtained after subtracting the obstacle would vanish, and the
above nondegeneracy argument could not be obtained.
\end{remark}
}

} 

\section*{Appendix}
\addcontentsline{toc}{section}{Appendix}

In the Appendix,
we {
explain how the optimal stopping problem \eqref{V_ch} motivates the
obstacle formulation \eqref{eqV*}
} and
verify that the {strong }solution of \eqref{eqV*}
{
constructed in Section~\ref{sec:3} coincides with the value function in
}
\eqref{V_ch}.
Furthermore, we provide the proof of
Lemma~\ref{lem:inequalityCP}, and show that $\mathcal F$ in the proof
of Theorem~\ref{thm:existenceVnepsilon} satisfies the conditions of
the Schauder fixed point theorem.

\renewcommand{\thesubsection}{\Alph{subsection}}
\setcounter{subsection}{0}
\renewcommand{\theHsubsection}{appendix.\Alph{subsection}}
\numberwithin{equation}{subsection}
\renewcommand{\theequation}{\thesubsection.\arabic{equation}}
\renewcommand{\theHequation}{appendix.\thesubsection.\arabic{equation}}

\subsection{Formulation of the model}\label{Appendix.A}

{
The stochastic problem \eqref{V_ch} motivates the obstacle problem
\eqref{eqV*}. We use the following order of proof: Section~\ref{sec:3}
constructs a strong solution of \eqref{AC} analytically, and
Appendix~\ref{Appendix.B} identifies its inverse transform with
\eqref{V_ch}. Thus, no It\^o formula is applied to an optimal stopping
value before its Sobolev regularity has been established.

Equivalently, the usual Snell-envelope/dynamic-programming argument first
characterizes the stochastic value as a viscosity solution; this route
uses smooth test functions, not derivatives of the value itself
(see \cite{PAS11,Rapuch2005}). We do not need that alternative
characterization for the strong-solution construction or for verification.
}

\subsection{Verification}\label{Appendix.B}

{
Write $v(t,s):=V(T-t,\log s)$ for the inverse transform of the
solution constructed in Theorem~\ref{thm:strong-solution}; we identify
$v$ with $V^{\rm ch}$ only after verification. On every compact
subset of $(0,T)\times(0,\infty)$ the change of variables and its
inverse are smooth with bounded derivatives. Consequently
\[
 v\in W^{1,2}_{p,\mathrm{loc}}((0,T)\times(0,\infty))
 \quad\text{for every }1<p<\infty,
 \qquad
 v\in C([0,T]\times(0,\infty)).
\]
Theorem~\ref{thm:strong-solution} and Lemma~\ref{DV} give
\[
 0\leq v(t,s)\leq s+\mathcal K_1,
 \qquad |v_s(t,s)|\leq1,
 \qquad v(T,s)=J(0,\log s).
\]

We use the local Sobolev It\^o formula
\cite[Theorem~1, p.~122]{KRY80}: for $p>3$, a continuous
$W^{1,2}_p$ function on a space--time neighborhood satisfies
It\^o's formula along a one-dimensional diffusion stopped inside
that neighborhood, when its coefficients are bounded and smooth
there and its diffusion coefficient is bounded away from zero.
The drift derivatives are interpreted almost everywhere; the local
occupation estimate in this formula justifies their integrals.
Here these hypotheses hold on compact intervals of positive stock
prices because the coefficients are $(r-q)s$ and $\sigma s$.

All expectations below are under the risk-neutral measure;
$\mathbb E_{t,s}$ denotes expectation for the diffusion started at $s$
at time $t$. Fix $0<t<T$, $S_t=s>0$, and
$\theta\in\mathcal U_{t,T}$.
For $k$ sufficiently large and $0<\delta<T-t$, define
\[
 \tau_k:=\inf\{u\geq t:S_u\notin(k^{-1},k)\}\wedge T,
 \qquad \theta_{k,\delta}:=\theta\wedge\tau_k\wedge(T-\delta).
\]
Applying the formula on a slightly larger compact cylinder gives
\begin{equation}\label{V:Ito verification}
\begin{aligned}
 e^{-r(\theta_{k,\delta}-t)}
 v(\theta_{k,\delta},S_{\theta_{k,\delta}})
 &=v(t,s)+\int_t^{\theta_{k,\delta}}
       e^{-r(u-t)}D v(u,S_u)\,du\\
 &\quad+\int_t^{\theta_{k,\delta}}
       e^{-r(u-t)}\sigma S_uv_s(u,S_u)\,dW_u,
 \qquad Dv:=v_t+\mathfrak L v\leq0.
\end{aligned}
\end{equation}

We verify the bounds needed to remove both stops. The stochastic
integral is square integrable, since
\begin{equation}\label{Condition:L2}
 \mathbb E_{t,s}\int_t^T
    e^{-2r(u-t)}\sigma^2S_u^2|v_s(u,S_u)|^2\,du
 \leq\sigma^2\mathbb E_{t,s}\int_t^T S_u^2\,du<\infty.
\end{equation}
The last inequality follows from the finite second moments of
geometric Brownian motion. The drift has the integrable bound
\[
 |Dv(t,s)|\leq qs+r(K_p+K_c)\quad\text{a.e.}
\]
Indeed it is zero in continuation; on the two contact regions,
Lemma~\ref{FB1} and the Sobolev level-set identities give respectively
$Dv=qs-rK_p$ and $Dv=rK_c-qs$ almost everywhere. Thus its absolute
value is dominated along the diffusion by an integrable function.

Almost surely, $S$ has a positive minimum and a finite maximum on
$[t,T]$, and therefore $\tau_k\uparrow T$. First let $k\to\infty$
in \eqref{V:Ito verification}, and then let $\delta\downarrow0$.
Continuity of $v$, its linear growth, and
$\mathbb E_{t,s}\sup_{t\leq u\leq T}S_u^2<\infty$ give $L^1$
convergence of the terminal terms. Dominated convergence handles
the drift terms, while \eqref{Condition:L2} and the It\^o isometry
give $L^2$ convergence of the stochastic integrals. Hence the same
formula holds with $\theta$ in place of $\theta_{k,\delta}$.
Taking expectations and using $v\geq\max\{C^A,P^A\}$ yields
\[
 v(t,s)\geq
 \mathbb E_{t,s}\!\left[e^{-r(\theta-t)}
     \max\{C^A(\theta,S_\theta),P^A(\theta,S_\theta)\}\right].
\]

Let
\[
 \theta^*:=\inf\{u\in[t,T]:
       v(u,S_u)=\max\{C^A(u,S_u),P^A(u,S_u)\}\}\wedge T.
\]
The contact set is closed and the paths are continuous, so this is
a stopping time. Before $\theta^*$ the equation is homogeneous,
and at $\theta^*$ the solution equals the obstacle (including
$\theta^*=T$). More explicitly,
\begin{equation}\label{V:C reg}
 Dv(u,S_u)=0\quad(t<u<\theta^*),\qquad
 v(\theta^*,S_{\theta^*})
 =\max\{C^A(\theta^*,S_{\theta^*}),P^A(\theta^*,S_{\theta^*})\}.
\end{equation}
The same localized formula, with $\theta=\theta^*$, now has zero
drift. After the two limiting steps just justified, it gives
\[
 v(t,s)=\mathbb E_{t,s}\!\left[e^{-r(\theta^*-t)}
     \max\{C^A(\theta^*,S_{\theta^*}),P^A(\theta^*,S_{\theta^*})\}\right].
\]
Taking the supremum over arbitrary $\theta$ in the preceding
inequality identifies $v$ with the value in \eqref{V_ch} and proves
optimality of $\theta^*$. For an initial time $t=0$, apply the formula
on $[\eta,\theta\vee\eta]$ and let $\eta\downarrow0$;
the same bounds and continuity justify the limit. The conditional
formulation follows by the Markov property, or by taking conditional
expectations throughout. We may thus write $v=V^{\rm ch}$.
}

\label{rev:verification-end}
\subsection{Proof of Lemma \ref{lem:inequalityCP}}\label{Appendix.C}

Similar results and approaches used in the proof can be found in previous
research; see, e.g., \cite{ZF09, ZF09-1}.

{
\subsubsection*{Proof of Part (i)}

The lower bounds follow immediately from the obstacle conditions.
For the upper bound of $\widehat{C}^A$, set
$W_c:=\widehat{C}^A-e^x$. 
Suppose the set $\{W_c>0\}$ is nonempty.
We then have
\[
    \widehat{C}^A>e^x\geq(e^x-K_c)^+,
\]
and hence $\widehat{C}^A$ satisfies the continuation equation in such region.
Therefore,
\[
    (\partial_\zeta-\mathcal L)W_c=-qe^x\leq0
    \quad\text{in }\{W_c>0\}.
\]
Since $W_c(0,x)=(e^x-K_c)^+-e^x\leq0$ and $W_c$ satisfies the exponential growth condition, the comparison principle
on unbounded domains \cite[Theorem 1.1]{YY08} yields $W_c\leq0$ in $\{W_c>0\}$, which is a contradiction.
Thus,
\[
    0\leq\widehat{C}^A\leq e^x
    \quad\text{in }\Omega_{T_c}.
\]

Similarly, set $W_p:=\widehat{P}^A-K_p$. On $\{W_p>0\}$,
$\widehat{P}^A>K_p\geq(K_p-e^x)^+$, so that
$\widehat{P}^A$ satisfies the continuation equation in such region. Hence
\[
    (\partial_\zeta-\mathcal L)W_p=-rK_p<0
    \quad\text{in }\{W_p>0\}.
\]
Since $W_p(0,x)=(K_p-e^x)^+-K_p\leq0$ and {$W_p$} satisfies the exponential growth condition, the same comparison
principle gives {a} contradiction and hence $W_p\leq0$. Therefore,
\[
    0\leq\widehat{P}^A\leq K_p
    \quad\text{in }\Omega_{T_p}.\]
}

{
\subsubsection*{Proof of Part (ii)}

Fix $h>0$ and write
\[
    d_h(x):=e^{x+h}-e^x.
\]
Since the operator $\mathcal L$ is invariant under spatial
translations, $\widehat{C}^A(\zeta,x+h)$ satisfies the same obstacle
problem with obstacle $(e^{x+h}-K_c)^+$. Since
$(e^{x+h}-K_c)^+\geq(e^x-K_c)^+$, the comparison principle for
variational inequalities \cite[Theorem 1.2]{YY08} gives
\[
    \widehat{C}^A(\zeta,x+h)\geq\widehat{C}^A(\zeta,x).
\]
On the initial boundary,
\[
    (e^{x+h}-K_c)^+
    \leq (e^x-K_c)^++d_h(x),
\]
while
$(\partial_\zeta-\mathcal L)d_h=qd_h\geq0$.
Hence, applying the
comparison principle \cite[Theorem 1.2]{YY08} yields
\[
    0\leq
    \widehat{C}^A(\zeta,x+h)-\widehat{C}^A(\zeta,x)
    \leq e^{x+h}-e^x.
\]
Dividing by $h$ and letting $h\to0^+$ and the continuity of $\partial_x\widehat{C}^A$, induced by the Sobolev embedding theorem yields
\[
    0\leq\partial_x\widehat{C}^A\leq e^x
    \quad\text{in }\Omega_{T_c}.
\]

For $\widehat{P}^A$, the inequality
$(K_p-e^{x+h})^+\leq(K_p-e^x)^+$ gives
\[
    \widehat{P}^A(\zeta,x+h)
    \leq\widehat{P}^A(\zeta,x).
\]
Moreover,
\[
    (K_p-e^x)^+
    \leq (K_p-e^{x+h})^++d_h(x),
\]
while
$(\partial_\zeta-\mathcal L)d_h=qd_h\geq0$.
Since both the obstacle and the initial value of
$\widehat{P}^A(\zeta,x+h)+d_h(x)$ are larger than those of
$\widehat{P}^A(\zeta,x)$, applying the comparison principle
\cite[Theorem 1.2]{YY08} yields
\[
    -\bigl(e^{x+h}-e^x\bigr)
    \leq
    \widehat{P}^A(\zeta,x+h)-\widehat{P}^A(\zeta,x)
    \leq0.
\]
Dividing by $h$ and letting $h\to0^+$ and the continuity of $\partial_x \widehat{P}^A$, induced by the Sobolev embedding theorem yields
\[
    -e^x\leq\partial_x\widehat{P}^A\leq0
    \quad\text{in }\Omega_{T_p}.
\]}

{
\subsubsection*{Proof of Part (iii)}

Fix $\delta>0$. Since the obstacle in \eqref{CO} is independent of
$\zeta$, the function $\widehat{C}^A(\zeta+\delta,x)$ satisfies the
same obstacle problem as $\widehat{C}^A(\zeta,x)$ in
$\Omega_{T_c-\delta}$, but with initial value
$\widehat{C}^A(\delta,x)$. By the obstacle condition,
\[
    \widehat{C}^A(\delta,x)
    \geq(e^x-K_c)^+
    =\widehat{C}^A(0,x).
\]
Hence the comparison principle for variational inequalities \cite[Theorem 1.2]{YY08} gives
\[
    \widehat{C}^A(\zeta+\delta,x)
    \geq\widehat{C}^A(\zeta,x)
    \quad\text{in }\Omega_{T_c-\delta}.
\]
Since $\delta>0$ is arbitrary,
$\partial_\zeta\widehat{C}^A\geq0$ almost everywhere in
$\Omega_{T_c}$. By \cite[Theorem 1.2]{Blanchet2006-1},
$\partial_\zeta\widehat{C}^A$ is continuous, and hence the inequality
holds throughout $\Omega_{T_c}$.

The same argument applies to $\widehat{P}^A$. Indeed,
$\widehat{P}^A(\delta,x)\geq(K_p-e^x)^+=\widehat{P}^A(0,x)$, and thus
\[
    \widehat{P}^A(\zeta+\delta,x)
    \geq\widehat{P}^A(\zeta,x)
    \quad\text{in }\Omega_{T_p-\delta}.
\]
It follows that $\partial_\zeta\widehat{P}^A\geq0$ almost everywhere
in $\Omega_{T_p}$, and another application of
\cite[Theorem 1.2]{Blanchet2006-1} gives
\[
    \partial_\zeta\widehat{P}^A\geq0
    \quad\text{in }\Omega_{T_p}.
\]
}

{
\subsubsection*{Proof of Part (iv)}

We first show that $\widehat C^A$ is strictly increasing and
$\widehat P^A$ is strictly decreasing with respect to $x$.
In the interior of the contact set of $\widehat C^A$, we have
$\widehat C^A=e^x-K_c$ and hence
$\partial_x\widehat C^A=e^x>0$.
By the $W^{1,2}_{p,\mathrm{loc}}$ regularity with $p>3$ and the
parabolic Sobolev embedding theorem,
$\partial_x\widehat C^A$ is continuous, so the same identity extends
to the free boundary. In the continuation region,
$\partial_x\widehat C^A\geq0$ by Part (ii), and
\[
    (\partial_\zeta-\mathcal L)
    \partial_x\widehat C^A=0.
\]
Thus, the strong maximum principle \cite[Corollary 2.4]{FRA04} yields $\partial_x\widehat C^A>0$ in the continuation region.
Combining all,
\[
    \partial_x\widehat C^A>0
    \quad\text{in }\Omega_{T_c}.
\]
By the same argument, we similarly obtain
\[
    \partial_x\widehat P^A<0
    \quad\text{in }\Omega_{T_p}.
\]

Consequently, for every fixed $\tau\in(0,T]$,
$x\mapsto C(\tau,x)-P(\tau,x)$ is strictly increasing.
By Part (i), for all sufficiently negative $x$,
$C(\tau,x)\leq e^x<K_p-e^x\leq P(\tau,x)$, while for all
sufficiently positive $x$,
$C(\tau,x)\geq e^x-K_c>K_p\geq P(\tau,x)$.
Therefore, there exists a unique $x_\tau\in\mathbb R$ such that
\[
    C(\tau,x_\tau)=P(\tau,x_\tau),
\]
and
\[
    C(\tau,x)<P(\tau,x)
    \quad\text{for }x<x_\tau,
    \qquad
    C(\tau,x)>P(\tau,x)
    \quad\text{for }x>x_\tau.
\]}
{
Moreover, $\partial_x(C-P)$ is continuous by the parabolic Sobolev
embedding theorem, while $\partial_\tau(C-P)$ is continuous by
Part~(iii). Since
\[
    \partial_x(C-P)(\tau,x_\tau)>0,
\]
the implicit function theorem gives
\[
    x_\tau\in C^1((0,T)).
\]
}

\subsection{Conditions of the operator \texorpdfstring{$\mathcal{F}$}{F} in Theorem \ref{thm:existenceVnepsilon}}\label{Appendix.D}

Let us prove that the operator $\mathcal{F}:\mathcal{A} \to \mathcal{X}$ in the proof of Theorem~\ref{thm:existenceVnepsilon} satisfies 
\begin{itemize}
    \item[\textbf{(1)}] $\mathcal{F}(\mathcal{A})\subset \mathcal{A}$,
    \item[\textbf{(2)}] $\mathcal{F}$ is continuous,
    \item[\textbf{(3)}] $\mathcal{F}(\mathcal{A})$ is precompact in $\mathcal{X}$.
\end{itemize}

\medskip
\noindent\textbf{(1)}
Fix any $w\in \mathcal{A}$ and consider $u=\mathcal{F}(w).$ 
Note that $J_{\varepsilon}(\tau,x)\geq 0$ and
$-\beta_\varepsilon(\cdots)\geq 0$.
{
Moreover, for sufficiently large $n\in\mathbb{N}$,
\[
    u(\tau,-n)=K_p-e^{-n}\geq0,
    \qquad
    u(\tau,n)=e^n-K_c\geq0
    \quad\text{for }\tau\in[0,T].
\]
}
Thus, by the comparison principle \cite[Theorem 1.1]{YY08}, we have $u\geq0.$

\medskip
\noindent\textbf{(2)}
To obtain the continuity of $\mathcal{F},$ it suffices to prove the following:
\begin{align*}
\text{Let} \ \{w_j\}\  &\text{be a sequence in}\
\mathcal{A}\ \  \text{such that} \ 
\displaystyle{\lim_{ j \to \infty}w_j}=w.\\
&\text{Then}\  
\displaystyle{\lim_{j \to \infty}\mathcal{F}(w_j)}=\mathcal{F}(w).
\end{align*}
Let $u_j:=\mathcal{F}(w_j)$ for each $j\in \mathbb{N}$ and
$u:=\mathcal{F}(w).$
By subtracting the two equations that $u$ and $u_j$ satisfy respectively, we have

{
\begin{equation*}
    \begin{cases}
        \partial_\tau (u-u_j)-\mathcal{L}(u-u_j)
        =-\beta_\varepsilon'(\cdots)(w-w_j)
        & \text{in }\Omega_T^n,
        \\
        (u-u_j)(\tau,-n)=0,
        \qquad
        (u-u_j)(\tau,n)=0
        & \text{for }\tau\in[0,T],
        \\
        (u-u_j)(0,x)=0
        & \text{for }x\in(-n,n).
    \end{cases}
\end{equation*}
}
Then the $W^{1,2}_p$-estimate \cite[Theorem 9.1, p.341]{LSU68} combined with the embedding theorem implies that 
\[
    \Vert u-u_j \Vert_{C(\overline{\Omega_T^n})}
    \leq
    \mathcal{K}
    \Vert\beta_\varepsilon'(\cdots)(w-w_j)
    \Vert_{L^\infty(\Omega_T^n)}
\]
for some constant $\mathcal{K}>0$.
Therefore, it follows that
$\displaystyle{\lim_{j \to \infty}u_j}= u.$

\medskip
\noindent\textbf{(3)}
Let $\{u_j\}_{j=1}^\infty$ be a sequence in the closure of
$\mathcal{F}(\mathcal{A})$ with respect to the
$\Vert\cdot\Vert_{C(\overline{\Omega_T^n})}$ norm.

\medskip
\noindent Case 1.
For each $u_j\in \mathcal{F}(\mathcal{A})$, there exists
$w_j\in\mathcal{A}$ such that $\mathcal{F}(w_j)=u_j$.

{
By applying the $W^{1,2}_p$ estimate on $u_j$, we obtain
\begin{align*}
\Vert u_j \Vert_{W^{1,2}_p(\Omega_T^n)}
\leq \mathcal{K}\Big(
&\Vert K_p-e^{-n}\Vert_{W^1_p(0,T)}
+\Vert e^n-K_c\Vert_{W^1_p(0,T)}
\\
&+\Vert J_{\varepsilon}(0,\cdot)\Vert_{W^{2}_p(-n,n)}
+\Vert \beta_\varepsilon(w_j-J_{\varepsilon})
\Vert_{L^p(\Omega_T^n)}
\Big)
\\
\leq \mathcal{K}\Big(
&\Vert K_p-e^{-n}\Vert_{W^1_p(0,T)}
+\Vert e^n-K_c\Vert_{W^1_p(0,T)}
\\
&+\Vert J_{\varepsilon}(0,\cdot)\Vert_{W^{2}_p(-n,n)}
+\Vert \beta_\varepsilon(-J_{\varepsilon})
\Vert_{L^p(\Omega_T^n)}
\Big)
\end{align*}
}
for some constant $\mathcal{K}>0$ that does not depend on
$j\in\mathbb{N}$.
Note that the last inequality is due to the monotonicity of
$\beta_\varepsilon$ and $w_j\geq0$.
Thus, combined with the embedding theorem it can be seen that
\[
    \Vert u_j\Vert_
    {C^{\frac{\alpha}{2},\alpha}(\overline{\Omega_T^n})}
    \le \mathcal{K}'
\]
for some $\alpha\in(0,1)$ and constant $\mathcal{K}'$ that does not
depend on $j\in\mathbb{N}$.

\medskip
\noindent Case 2.
For each fixed $u_i\notin \mathcal{F}(\mathcal{A})$,
there exists a sequence $\{u^{(i)}_l\}_{l=1}^\infty
\subset\mathcal{F}(\mathcal{A})$ such that
\[
    u^{(i)}_l \rightarrow u_i
    \quad \text{in }\ C(\overline{\Omega_T^n})
    \quad \text{as }\ l\rightarrow\infty.
\]
Hence, there exists $\{w^{(i)}_l\}_{l=1}^\infty$ in $\mathcal{A}$
such that $\mathcal{F}(w^{(i)}_l)=u^{(i)}_l$.
Again, by the $W^{1,2}_p$ estimate, the embedding theorem, and the
monotonicity of $\beta_\varepsilon$, we deduce
\[
    \Vert u_l^{(i)}\Vert_
    {C^{\frac{\alpha}{2},\alpha}(\overline{\Omega_T^n})}
    \leq \mathcal{K}'
\]
for some $\alpha\in(0,1)$.
Take $l\rightarrow\infty$ to get
\[
    \Vert u_i\Vert_
    {C^{\frac{\alpha}{2},\alpha}(\overline{\Omega_T^n})}
    \leq \mathcal{K}'.
\]

Combining both cases, we conclude
\[
    \Vert u_j\Vert_
    {C^{\frac{\alpha}{2},\alpha}(\overline{\Omega_T^n})}
    \leq \mathcal{K}'
\]
for all $j\in\mathbb{N}$.

By the Arzel\`a--Ascoli theorem, there exists a subsequence
$\{u_{j_k}\}_{k=1}^\infty \subset \{u_j\}_{j=1}^\infty$ and
$u\in C(\overline{\Omega_T^n})$
such that
\[
    u_{j_k}\rightarrow u
    \quad \text{in }\ C(\overline{\Omega_T^n})
    \quad \text{as }\ k\rightarrow\infty.
\]
Observe that $u$ is contained in the closure of
$\mathcal{F}(\mathcal{A})$ by the assumption.
This proves that $\mathcal{F}(\mathcal{A})$ is precompact in
$\mathcal{X}=C(\overline{\Omega_T^n}).$
\qed

\section*{Acknowledgment}
J. Jeon was supported by NRF grant funded by MSIT (RS-2023-00212648). J.Ok was supported by NRF grant funded by MSIT (NRF-2022R1C1C1004523).

\end{document}